\documentclass[a4paper]{amsart}
\numberwithin{equation}{section}

\usepackage{amssymb}
\usepackage{mathrsfs}
\usepackage{enumerate, xspace}
\usepackage{enumitem}
\usepackage{marvosym} %monetary symbols, \EUR or \EURcr, \EyesDollar
\usepackage[sans]{dsfont}        %allows \mathds{E 1}, without sans is less heavy
\usepackage{upgreek}
\usepackage[backgroundcolor=blue!20!white%
            ,linecolor=blue!20!white,%
            ,textsize=footnotesize%
            ]{todonotes}
\usepackage[colorlinks]{hyperref}
\usepackage{verbatim}
\usepackage{bm}
\usepackage[all]{xy}
\ifx\XeTeXversion\undefined
\usepackage[english]{babel}
\else
\usepackage{fontspec}
\usepackage{polyglossia}
\setmainlanguage{english}
\usepackage{microtype}
\fi

\makeatletter

\newcommand\mergedsub[2]{#1\sc@sub{#2}}
\newcommand\newsubcommand[3]{\newcommand#1{\mergedsub{#2}{#3}}}
\def\sc@sub#1{\def\sc@thesub{#1}\@ifnextchar_{\sc@mergesubs}{_{\sc@thesub}}}
\def\sc@mergesubs_#1{_{\sc@thesub#1}}

\newcommand\mergedsubC[2]{#1\sc@subC{#2}}

\def\sc@subC#1{\def\sc@thesub{#1}\@ifnextchar_{\sc@mergesubsC}{_{\sc@thesub}}}
\def\sc@mergesubsC_#1{_{#1,\sc@thesub}}

\newcommand\mergedsup[2]{#1\sc@sup{#2}}

\def\sc@sup#1{\def\sc@thesup{#1}\@ifnextchar^{\sc@mergesups}{^{\sc@thesup}}}
\def\sc@mergesups^#1{^{\sc@thesup,#1}}

\newcommand\mergedsupC[2]{#1\sc@supC{#2}}

\def\sc@supC#1{\def\sc@thesup{#1}\@ifnextchar^{\sc@mergesupsC}{^{\sc@thesup}}}
\def\sc@mergesupsC^#1{^{#1,\sc@thesup}}

\newcommand\mergedpar[2]{#1\sc@parC{#2}}

\def\sc@parC#1{\def\sc@thepar{#1}\@ifnextchar_{\sc@mergeparsub}{\@ifnextchar'{\sc@mergeparprime}{(\sc@thepar)}}}
\def\sc@mergeparsub_#1{_{#1}(\sc@thepar)}
\def\sc@mergeparprime#1{#1(\sc@thepar)}

\makeatother

\newtheorem{thm}{Theorem}[section]
\newtheorem{lem}[thm]{Lemma}
\newtheorem{cor}[thm]{Corollary}

\newtheorem{prop}[thm]{Proposition}

\newtheorem{rem}[thm]{Remark}

\newcommand\cB{{\mathcal B}}
\newcommand\cC{{\mathcal C}}

\newcommand\cF{{\mathcal F}}
\newcommand\cG{{\mathcal G}}
\newcommand\cH{{\mathcal H}}
\newcommand\cI{{\mathcal I}}

\newcommand\cO{{\mathcal O}}
\newcommand\cP{{\mathcal P}}
\newcommand\cM{{\mathcal M}}

\newcommand\cU{{\mathcal U}}

\newcommand\bE{{\mathbb E}}
\newcommand\bF{{\mathbb F}}
\newcommand\bG{{\mathbb G}}

\newcommand\bL{{\mathbb L}}
\newcommand\bN{{\mathbb N}}
\newcommand\bP{{\mathbb P}}

\newcommand\bR{{\mathbb R}}

\newcommand\bT{{\mathbb T}}

\newcommand\bZ{{\mathbb Z}}

\newcommand\fkL{{\mathfrak{L}}}

\newcommand\ve{\varepsilon}

\newcommand\vf{\varphi}

\newcommand\nc[1]{_{\cC^{#1}}}

\newcommand\bRp{\bR_+}
\newcommand\intr{\textup{int}\,}

\newcommand{\tE}{{\tilde E}}
\newcommand{\tF}{{\tilde F}}
\newcommand{\tM}{{\tilde M}}
\newcommand{\tz}{{\tilde z}}

\setlist[enumerate,1]{label=\textup{(\alph*)}}
\newcommand{\efrac}[2]{#1/#2} %% fraction at the exponent

\newcommand\Id{{\mathds{1}}}

\newcommand\stable{{s}}  %%cambiato da \sigma, perche' \sigma e' gia strausato
\newcommand{\drifte}{v_\ve}

\newcommand{\Const}{{C_\#}}
\newcommand{\const}{{c_\#}}
\newcommand{\vei}{\ve^{-1}}
\newcommand{\veh}{\ve^{1/2}}
\newcommand{\deh}{\textup{d}}
\newcommand{\st}{\,:\,}

\renewcommand{\ln}{\textup{log}\,}
\newcommand{\dist}{\textup{dist}}

\newcommand{\supp}{\textup{supp}\,}
\newcommand{\pint}[1]{{\lfloor#1\rfloor}}

\newsubcommand{\npiA}{\boldsymbol\pi}{\alphaset}

\newcommand{\alphaset}{\mathcal{A}}

\newcommand{\fpath}{h}

\newcommand{\ho}{\hat\omega}
\newcommand{\thetasl}{\theta^*_\ell}
\newcommand{\zeroes}{{\boldsymbol z}}

\newcommand{\Cn}{\textup{A}}
 \newsubcommand{\TCn}{T}{\Cn}

\newcommand{\Leb}{\textup{Leb}}
\newcommand{\fm}{p}

\newcommand{\expo}[1]{\exp(#1)}
\newcommand\shiftPar{\kappa}
\newcommand\deviation{\Delta}
\newcommand\Var{\boldsymbol{\upsigma}} %% troppi \Sigma
\newcommand\bVar{\boldsymbol{\hat\upsigma}}
\newcommand\Beta{\boldsymbol{\eta}}

\newcommand\slo{z}

\definecolor{reviewColor}{RGB}{255,172,255}

\newcommand\numToPrime[1]{%
  \ifnum0=#1\relax%
  \else%
  \ifnum1=#1\relax%
  '\else%
  \ifnum2=#1\relax%
  ''\else%
  \ifnum3p=#1\relax%
  '''\else%
  ^{(#1)}%
  \fi\fi\fi\fi} % This is quite a hack, but so far it works
\newcommand\dtt[1]{D_1\numToPrime{#1}}
\newcommand{\spc}[1]{c_#1}

\newcommand\dttrho[1]{D_2\numToPrime{#1}}
\newcommand{\stdf}{\fkL}
\newcommand{\stdpSet}[1]{L_{#1}}
\newcommand{\stdfSet}[2]{\bL_{#1}\ifx&#2&\else(#2)\fi}
\newcommand{\privateFell}{\ell}
\newcommand{\fell}{\privateFell}
\newcommand{\fellf}[1]{\mergedpar{\privateFell}{#1}}

\newcommand{\vol}{\textup{vol}\,}
\newcommand{\Jac}{\textup{Jac}\,}

\newcommand\sW{{\mathscr W}}

\begin{document}
\title[Fast-slow  partially hyperbolic systems]{Fast-slow partially hyperbolic systems versus Freidlin--Wentzell random systems}
\author[De Simoi]{Jacopo de Simoi}
\address{Jacopo De Simoi\\
  Department of Mathematics\\
  University of Toronto\\
  40 St George St. Toronto, ON M5S 2E4}
\email{{\tt jacopods@math.utoronto.ca}}
\urladdr{\href{http://www.math.utoronto.ca/jacopods/}{http://www.math.utoronto.ca/jacopods}}
\author[Liverani]{Carlangelo Liverani}
\address{Carlangelo Liverani\\
  Dipartimento di Matematica\\
  II Universit\`{a} di Roma (Tor Vergata)\\
  Via della Ricerca Scientifica, 00133 Roma, Italy.}
\email{{\tt liverani@mat.uniroma2.it}}
\urladdr{\href{http://www.mat.uniroma2.it/~liverani/}{http://www.mat.uniroma2.it/~liverani}}
\author[Poquet]{Christophe Poquet}
\address{Christophe Poquet\\
  Universit\'e de Lyon, Universit\'e Lyon 1, Institut Camille Jordan, UMR 5208\\
  43 boulevard du 11 novembre 1918\\
  F-69622 Villeurbanne, France}
  \email{{\tt poquet@math.univ-lyon1.fr}}
  \urladdr{\href{http://math.univ-lyon1.fr/~poquet/}{http://math.univ-lyon1.fr/~poquet/}}
\author[Volk]{Denis Volk}
\address{Denis Volk\\
  Centre for Cognition and Decision making, National Research University Higher School of Economics, Russian Federation}
\email{{\tt dvolk@hse.ru}}
\urladdr{\href{http://tinyurl.com/DenisVolk}{http://tinyurl.com/DenisVolk}}
%\date{\today. {\bf File: {\jobname}.tex.}}% eliminate in final version

\thanks{This work has been supported by the European Advanced Grant Macroscopic Laws and Dynamical Systems (MALADY)
  (ERC AdG 246953). D.V. has been partially funded by the Russian Academic Excellence Project '5-100'}

\begin{abstract}
We consider a simple class of fast-slow partially hyperbolic dynamical systems and show that the (properly rescaled) behaviour of the slow variable is very close to a Freidlin--Wentzell type random system for times that are rather long, but much shorter than the metastability scale.  Also, we show the possibility of a ``sink" with all the Lyapunov exponents positive, a phenomenon that turns out to be related to the lack of absolutely continuity of the central foliation.
\end{abstract}

\keywords{Averaging, metastability, partially hyperbolic, decay of correlations}
\subjclass[2000]{37A25, 37C30, 37D30, 37A50, 60F17}

\maketitle
\section{Introduction}
In \cite{DeL0, DeL1,DeL2} the first two authors studied the following class of partially hyperbolic systems of the fast-slow type on $\bT^2$
\begin{equation}\label{eq:f-eps}
F_\ve(x,\theta)=(f(x,\theta),\theta+\ve\omega(x,\theta))\mod 1,
\end{equation}
with $\ve>0$, small, $F_\ve\in\cC^5(\bT^2,\bT^2)$, and $\inf_{x,\theta}\partial_x f(x,\theta)\geq \lambda>1$, $\|\omega\|_{\cC^4}=1$.\footnote{ In fact in such papers it was assumed only $F_\ve\in\cC^4(\bT^2,\bT^2)$, here we need a bit more regularity.}
As usual it is important to specify the type of initial conditions under which we like to study the dynamical systems $(\bT^2,F_\ve)$. It is well known that, in order to be able to obtain meaningful results for long times, they must be random. More precisely, if we define $(x_n,\theta_n)=F_\ve^n(x_0,\theta_0)$, then we would like to consider, at least, the initial condition $\theta_0\in\bT^1$ fixed, while $x_0\in\bT^1$ is distributed according to a probability measure with smooth density w.r.t. Lebesgue. Then $(x_n,\theta_n)$ can be viewed as a (Markov) random process.\footnote{ Admittedly a rather degenerate Markov process as the transition kernel is singular.}

We refer to the introductions of the above mentioned papers for a lengthy discussion of the relevance of such systems, the connection with averaging, homogenisation theory, metastability and statistical mechanics as well as for a discussion of the related literature.

Even though \eqref{eq:f-eps} is arguably the simplest possible model problem for a fast--slow partially hyperbolic system, its exact properties are not understood in full generality. If we want to develop a general theory for fast-slow partially hyperbolic systems, it is then important to see where do we stand and what are the open problems for the above basic model.

Probably the most striking fact concerning the dynamical systems $(\bT^2, F_\ve)$, and more generally fast-slow systems, is that they have many different relevant time scales. More precisely there exists some parameters $\alpha_0, c_0, c_1>0$:
\begin{itemize}
\item{\bf Initial times:} If $n< c_1\ln\ve^{-1}$, then the time is so short that the statistical property plays no
  significant role and one can, in principle, compute the trajectory numerically with arbitrary precision starting from
  a deterministic initial condition.
\item{\bf Short times:} If $c_1\ln\ve^{-1}<n< \ve^{-1+\alpha_0}$, then the $x$ variable is, essentially, distributed according to the invariant measure of $f(\cdot,\theta_0)$ while $|\theta_n-\theta_0|\leq \ve^{\alpha_0}$. Thus $\theta$ appears to be almost a constant of motion.
\item{\bf Long times:} If $\ve^{-1+\alpha_0}< n< \ve^{-1-\alpha_0}$, the evolution of the variable $\theta_{\ve^{-1}t}$ is close (in a precise technical sense) to a random process described by a stochastic differential equation.
\item{\bf Very long times:} If $\ve^{-1-\alpha_0}< n< e^{c_0\ve^{-1}}$, the system may behave like if several invariant SRB measures exist ({\em metastable state}).
\item{\bf Arbitrarily long times:} If  $n> e^{c_0\ve^{-1}}$, finally the system exhibits its true statistical properties (SRB measures, decay of correlations, etc.).
\end{itemize}

The first regime poses interesting problems in the fields of numerical analysis, but we will not discuss them here. The
second regime can be studied by applying standard results on the decay of correlations and we will not discuss it either,
as it can be seen as a special case of the third. The third regime is the one on which most of this paper will focus. We
will see that a precise understanding of this regime gives relevant informations also for longer times. The only other
discussions involving the behaviour of the system for longer times will be our discussion of the central foliation, that
obviously contains informations on the infinite time dynamics, although only of a very local nature. As for the last regime, we will only mention briefly the standing open problems.

Let us discuss the last three regimes in a bit more detail.
\subsection{Long Times}\ \newline
In this regime it is useful to rescale time, so we define
\begin{equation}\label{eq:thetaproc}
\theta_\ve(t)=\theta_{\lfloor \ve^{-1} t\rfloor}+(\ve^{-1}t-\lfloor \ve^{-1} t\rfloor)(\theta_{\lfloor \ve^{-1} t\rfloor+1}-\theta_{\lfloor \ve^{-1} t\rfloor}).
\end{equation}
Also it is convenient to see $\theta_\ve$ as a random variable in $\cC^0(\bRp,\bT)$.  In the paper \cite{DeL1} it is
shown that, in any interval $[0,T]$, $\theta_\ve$ converges weakly as $\ve\to 0$ to the solution of
\begin{align}\label{eq:averageeq}
  \frac{\deh\bar\theta}{\deh t}&=\bar\omega(\bar\theta)&
  \bar\theta(0)&=\theta_0,
\end{align}
where $\bar\omega(\theta)=\mu_\theta(\omega(\cdot,\theta))$, $\mu_\theta$ being the unique SRB measure of $f(\cdot,\theta)$, see \cite{Babook}.  For future use, let us also define the
function $\ho(x,\theta)=\omega(x,\theta)-\bar\omega(\theta)$.  Note that, by the
differentiability of $\mu_\theta$ with respect to $\theta$ (see~\cite[Section 8]{GL06}) we have that $\bar\omega\in\cC^{4-\alpha}$, for each $\alpha>0$.
Thus \eqref{eq:averageeq} is a well defined differential equation. Also \cite{DeL0} contains a results on the fluctuations: that is, if we define $\zeta_\ve(t)=\ve^{-\efrac 12}(\theta_\ve(t)-\bar\theta(t))$, then, in any time interval $[0,T]$, $\zeta_\ve$ converges to $\zeta$, defined by
\begin{equation}\label{eq:diffusion0}
\begin{split}
&d\zeta=\bar\omega'(\bar\theta)\zeta dt +\bVar(\bar \theta)dB\\
&\zeta(0)=0,
\end{split}
\end{equation}
where $B$ is the standard $1$-dimensional Brownian motion and the diffusion coefficient $\bVar$ is given
by the Green-Kubo formula
\begin{equation}\label{e_definitionBarChi}
\begin{split}
  \bVar(\theta)^2 =& \mu_\theta\left(\hat\omega(\cdot,\theta)\hat\omega(\cdot,\theta)\right)+ 2\sum_{m=1}^{\infty}
   \mu_\theta\left( \hat\omega(f_\theta^m(\cdot),\theta)\hat\omega(\cdot,\theta)\right),
\end{split}
\end{equation}
where we have used the notation $f_\theta(x)=f(x,\theta)$.  In addition, $\bVar(\theta)$ is differentiable (see~\cite[Section 8]{GL06} again) and it is strictly positive, unless $\hat\omega(\theta, \cdot)$ is a coboundary for $f_\theta$, see \cite{liverani-clt}. Thus, from now on we will assume:
\begin{enumerate}[label=\textup{(A\arabic*)},ref=(A\arabic*),resume]
\item\label{a_noCobo} for each $\theta\in\bT$, the function $\omega(\cdot,\theta)$ is not
  cohomologous to a constant function with respect to $f_\theta$.
\end{enumerate}

However, a much sharper result is proven in \cite{DeL1}: a local limit theorem with error, see Theorem  \ref{thm:lclt}  for details.

Here we go one step further and we prove that $\theta_\ve$ is very close (in a technical sense to be specified later) to the following Freidlin-Wentzell type process for times of order $\ve^{-\alpha}$, for some $\alpha>0$,
\begin{equation}\label{eq:f-w}
\begin{split}
&\deh\Beta(t)=\bar\omega(\Beta(t))\deh t+\sqrt \ve\bVar(\Beta(t))\deh B\\
&\Beta(0)=\theta_0.
\end{split}
\end{equation}
The above equation has been extensively studied, starting with \cite{WeFr}, and is know to exhibit metastable states. Before considering longer times, it is convenient to understand very precisely the behaviour of \eqref{eq:f-w} in the present regime. To our surprise, we were not able to locate the needed results in the literature, so we provide them here. This will allow us to obtain a very precise description of our system in this time scale.

Note that such a result can be used to considerably simplify various arguments in \cite{DeL2}.
\subsection{Very Long Times}\ \newline
Up to now the only condition on $\bar \omega$ was that it is not a coboundary. It turns out that precise results in the very long time scale are known only in certain cases, namely when $\bar \omega$ has zeroes. This is due to the fact that, in such a case, equation \eqref{eq:averageeq} has, generically, attractive fixed points and then the dynamics tends to be localised. On the contrary, if $\bar \omega$ has no zeroes, then the average dynamics is, essentially, a rotation and its statistical properties need a longer time scale to manifest (see Lemma \ref{lem:conver-no-zero}). Thus let us assume
\begin{enumerate}[label=\textup{(A\arabic*)},,ref=(A\arabic*),resume]
  \item\label{a_discreteZeros} $\bar\omega$ has a non-empty discrete set of non-degenerate zeros.
\end{enumerate}
Note that the above is a generic condition, once the zeroes do exist. The previous mentioned results imply then that the dynamics spends most of the time in a $\sqrt\ve$ neighbourhood of the attractive fixed points of \eqref{eq:averageeq}, see Proposition \ref{prop: close to gaussians}. Indeed, the large deviation results of \cite{DeL1} imply that the probability of escaping one such sink, if possible at all, is exponentially small in $\vei$. This means that it will be necessary an exponentially long time for the distributions of the $\theta$ variable to change appreciably. That is, there are quasi stationary states (metastability).
The occurrence of metastable states for pure deterministic systems was first found, in a different but similar context, by Kifer \cite{Kifer09}, in which it is shown that the system visits the different metastable states essentially following a Markov chain. However, the results there do not suffice to investigate the true invariant measures of the system.

On the contrary, our estimates show that, if we consider any accumulation point of $\frac 1n\sum_{k=0}^{n-1} F^k_{\ve,*}\nu_0$, where $\nu_0$ is one of our initial probability distribution,\footnote{ $F_*$ stands for the pushforward, namely $F_*\mu(\vf)=\mu(\vf\circ F)$.} then such points must be very close to a convex combinations of the metastable states. As all the possible {\em physical} measures of the systems must belong to such accumulation set,\footnote{ Recall that $\mu$ is a physical measure if, for all continuous $g$, $\lim_{n\to\infty}\frac 1n\sum_{k=0}^{n-1}g\circ F_\ve^k(x)=\mu(g)$ for $x$ belonging to a set of positive Lebesgue measure.} see \cite[Lemma 9.8]{DeL2}; this provides detailed informations on the possible structure of the physical measures. In turn, this also allows to compute the Lyapunov exponents of the system.

Of course, the variable $x$ undergoes uniform expansion, hence one Lyapunov exponent is trivially positive. The other is the Lyapunov exponent in the central direction (the direction associated to the slow variable $\theta$).
Indeed, in Section~\ref{sec:central-fol} we will see that the maps $F_\ve$, and hence $F_0$, have an invariant center
foliation. Let $(s_*(x,\theta),1)$ be the vectors defining the center distribution of $F_0$. Such a distribution is known to be close to the one of $F_\ve$.  Let us denote with $\psi_*$
the directional derivative of $\omega$ in the center direction, or, more precisely, with
respect to the vector $(s_*(x,\theta),1)$, that is:
\begin{align}\label{e_definitionPsi}
  \psi_*(x,\theta) = \partial_x \omega(x,\theta) s_*(x,\theta) + \partial_\theta\omega(x,\theta).
\end{align}
It is convenient to define also the average $\bar \psi_*(\theta)=\mu_\theta(\psi_*(\cdot,\theta))$.
Essentially, $1+\ve\psi_* $ is the one step-contraction (or expansion) in the center direction.
Then, for each invariant measure $\mu$,  $\mu(\ln (1+\psi_*))$ is the central Lyapunov exponent associated to such a measure.

Naively, one could think that such an exponent is always negative, as it is driven by the contraction in the sinks in which the dynamics spends most of the time. Surprisingly, this is not always the case. This was  already conjectured in \cite{DeL2} and is proven here. In addition, we show that the presence of a positive Lyapunov exponent in the central foliation is associated to the foliation non being absolutely continuous. Such a pathology of the central foliation was already discovered in other examples, e.g. volume preserving partially hyperbolic maps \cite{SW, RW}, but here emerges in a totally natural and robust manner for systems whose invariant measure is not previously known and it is not constant.

\subsection{Arbitrarily Long Times}
In this regime, the system exhibits its true statistical properties, the first being the existence or not of physical measures.
They are proven to exist in the case of mostly expanding central direction~\cite{ALP05}, and mostly contracting central
direction~\cite{Do00, deCastro02}.  For central direction with zero Lyapunov
exponents (or close to zero) physical measures are known to exist generically \cite{Tsujii05}.
In the case of mostly expanding and mostly contracting central foliations the above papers contain also some information on uniqueness and mixing of the physical measure, but not of a quantitive nature.

Precise quantitive results beyond what already discussed are known only when the central Lyapunov exponent is negative.
More precisely, calling $\{\theta_{k,-}\}_{k=1}^\zeroes$ the zeroes of $\bar \omega$ such that $\bar\omega'(\theta_{k,-})<0$, the condition
\begin{enumerate}[label=\textup{(A\arabic*)},,ref=(A\arabic*),start = 2]
  \item\label{a_almostTrivial} $\max_{k\in\{1,\cdots,\zeroes\}}\bar \psi_*(\theta_{k,-}) <0$.
\end{enumerate}
implies that the center Lyapunov exponent is Lebesgue-a.s. negative (see \cite{DeL2} or Section \ref{sec:lyap}).
Under such a condition \cite{DeL2} obtained a full classification of the SRB measures and sharp estimates of the decay of correlations.

Numerical simulations and the results in \cite{ ABV00,ALP05,Gou06, Tsujii05} suggest that similar type of results should hold in much greater generality (in particular for positive Lyapunov or zero exponents).
Yet, to obtain similar results in the case of non negative Lyapunov exponent, or for the case in which $\bar\omega$ has no zeroes, stands as an important open challenge.

\tableofcontents

\section{Limit theorems: a recap}\label{ss_limitThms}
Standard pairs are a very convenient way to describe initial conditions which can be realised by interacting with a system in the past. The reader can find a basic introduction to standard pairs in \cite{DeL0} and a more precise description of the type of standard pairs needed in the present context in \cite{DeL1, DeL2}. Here we content ourselves with a super brief introduction, just to establish notations.

Let us fix a small $\delta>0$, and $\dtt0, \dtt1, \spc1>0$ large enough. Let us define the set of functions
\begin{align*}
  \Sigma_{\spc1}=\{G\in \cC^3([a,b], \bT)\st&
  a,b\in\bT, b-a\in[\delta/2,\delta],\\&
  \|G'\|\leq \ve \spc1,\, \|G''\|\leq \ve \dtt0 \spc1,\,\|G'''\|\leq \ve \dtt1 \spc1\}.
\end{align*}
We associate to any $G\in \Sigma_{\spc1}$ the map $\bG(x)=(x,G(x))$; the graph of any
such $G$ (i.e. the image of $\bG$) will be called a \emph{standard curve}.

Next, fix $\dttrho0, \spc2>0$ large enough.  We define the set of $\spc2$-\emph{standard} probability densities on the standard curve
$G$ as
\begin{align*}
  D_{\spc2}(G)=\left\{\rho\in \cC^2([a,b],\bRp)\st \int_a^b\rho(x)\deh x=1,\
\left\|\frac{\rho'}{\rho}\right\|\leq
  c_2,\,\left\|\frac{\rho''}{\rho}\right\|\leq \dttrho0\spc2\right\}.
\end{align*}
A {\em standard pair} $\ell$ is given by $\ell=(\bG,\rho)$, where
$G\in\Sigma_{\spc1}$ and $\rho\in D_{\spc2}(G)$.  To any standard pair
$\ell=(\bG,\rho)$ is uniquely associated a probability measure $\mu_\ell$ on $\bT^2$ defined
as follows: for any Borel-measurable function $g$ on $\bT^2$ let
\[
\mu_\ell(g):=\int_a^{b} g(\bG(x))\rho(x) \deh{x}.
\]
Let $\stdpSet{\spc1,\spc2}$ be the set of standard pairs. For each $\ell \in \stdpSet{\spc1,\spc2}$ we will use $a_\ell,b_\ell$, $\bG_\ell$ and $\rho_\ell$ for the domain, graph and density associated to the standard pair. Moreover, for future reference, given $\ell\in \stdpSet{\spc1,\spc2}$, let us define
\[
\thetasl=\int_{a_\ell}^{b_\ell} G_\ell(x)\rho_\ell(x) dx=\mu_\ell(\theta).
\]

A standard family can be conveniently regarded as a
\emph{random standard pair}.  More precisely: a \emph{standard family}
$\stdf$ is given by the triplet $(\alphaset,\cF,\fm)$ where $\alphaset$ is a countable set
$\cF$ is a map $\fell:\alphaset\to\stdpSet{\spc1,\spc2}$ and $\fm$ is a probability measure on $\alphaset$.

A $(\spc1,\spc2)$-standard family $\stdf$ identifies a unique probability measure $\mu_\stdf$
on $\bT^2$: for any
Borel-measurable function $g$ of $\bT^2$, let
\[
\mu_\stdf(g):=\int_\alphaset\mu_{\fellf{\alpha}}(g)\deh\fm.
\]
We will use $\stdfSet{(\spc1,\spc2)}{}$ to designate the set of standard families and ${\overline \bL}_{(\spc1,\spc2)}$
to designate the standard measures, that is, the weak closure of the measures associated to a standard family.

The basic properties of standard families rest in the following fact.
\begin{prop}[{\cite[Proposition 5.2]{DeL2}}]\label{p_invarianceStandardPairs}
There exist $\spc1$, $\spc2$ such that, if $\ve$ is sufficiently small and $\ell$ is a
standard pair, $F_{\ve*}\mu_{\ell}$ can be seen as the measure associated to a standard family.
\end{prop}

Accordingly, if our initial condition is expressed by a standard pair, then the pushforward of the initial measure will
always consists of a convex combination of standard pairs, that is it will be an element of $\bL_{(\spc1,\spc2)}$. On the other hand the physical measures are obtained as the limit of the pushforward of initial measures with densities with respect to Lebesgue. Since such measures can be approximated by standard families, the physical measures can be obtained as accumulation points of standard families, see \cite[Lemma 9.8]{DeL2} for details. Hence, all the physical measures will belong to ${\overline \bL}_{(\spc1,\spc2)}$.

For a given, but arbitrary, smooth function $\psi$, let us define the function $\zeta_n$ as:
\begin{equation} \label{e_definitionZ}%
  \zeta_n=\ve\sum_{k=0}^{n-1}\psi\circ F_\ve^k.
\end{equation}
Next, let
$\slo_n=(\theta_n,\zeta_n)$ and define the polygonal interpolation
\[
\slo_\ve(t)=%
\slo_{\pint{t\vei}} +%
(t\vei-\pint{t\vei})%
(\slo_{\pint{t\vei}+1}-\slo_{\pint{t\vei}}).
\]
For any $t\geq 0$ and $\theta_*\in\bT^1$, we define the function
$\bar\slo(t,\theta_*)$ to be the solution of the ODE
\begin{equation}\label{e_averagedSlo}
\begin{split}
& \frac{\deh}{\deh t}\bar\slo(t)=\left(\bar\omega(\bar\theta(t)),\bar\psi(\bar\theta(t))\right),\\
&\bar\slo(0)=(\theta_*,0)
\end{split}
\end{equation}
where $\bar\psi(\theta)=\mu_\theta(\psi(\cdot,\theta))$, $\bar\omega=\mu_\theta(\omega(\cdot,\theta))$ and $\mu_\theta$ is the unique SRB measure of the map $f_\theta(x)=f(x,\theta)$. We conveniently introduced functions $\theta_\ve,\,
\zeta_\ve$ and $\bar\theta,\,\bar\zeta$ so that $\slo_\ve=(\theta_\ve,\zeta_\ve)$ and
$\bar\slo=(\bar\theta,\bar\zeta)$.

Then (see \cite[Theorem 2.1]{DeL1}), as $\ve\to 0$, provided that the initial
conditions $(x_0,\theta_0)$ are distributed according to standard pairs such that $\thetasl=\theta_*$, the random variable $\slo_\ve(t)$ converges
in probability to $\bar\slo(t,\theta_*)$.  It is then natural to attempt a description of
the behavior of deviations from the averaged dynamics.  For any $p=(x_0,\theta_0)$, let
$\Delta\slo(t,p)=(\Delta\theta(t,p),\Delta\zeta(t,p)):=\slo_\ve(t,p)-\bar\slo(t,\theta_0)$.
In this respect, if $\ve$ is sufficiently small, we can obtain (see~\cite[Theorem 2.2,
Corollaries 2.3-5]{DeL1}) the following
\begin{thm}[{\cite[Proposition 2.3]{DeL1}}]\label{t_largeDevz}
  Fix $T>0$ and define, for $R_\theta,R_\zeta>\Const\sqrt\ve$:
  \[
  Q(R_\theta,R_\zeta)=\{p\in\bT^2\st \sup_{t\in[0,T]} |\Delta\theta(t,p)|\geq
  R_\theta\text{ or} \sup_{t\in[0,T]} |\Delta\zeta(t,p)|\geq R_\zeta \}.
  \]
  Then, for any standard pair $\ell$, we have
  $\mu_\ell(Q(R_\theta,R_\zeta))<\exp(-\const\vei \min(R_\theta^2,R_\zeta^2))$, where
  $\const$ is a constant which does not depend on $\ell$.
\end{thm}

We say that a differentiable path $\fpath$ of length $T$ is \emph{admissible} if for any $s\in[0,T]$,
$h'(s)\subset\intr\Omega(h(s))$, where for any $\theta\in\bT$, we define the (non-empty, convex and compact) set
\begin{align}\label{e_definitionOmega}
  \Omega(\theta)=\{\mu(\omega(\cdot,\theta))\,| \,\mu\text{ is a
  }f_\theta\text{-invariant probability}\}.
\end{align}

\begin{thm}[{\cite[Theorem 6.3]{DeL2}}]\label{l_largeDevzLowerBound}
  Let $\fpath\in\cC^1([0,T],\bR)$ be an admissible path joining $\theta_0$ to
  $\theta_1$; then, provided $\ve$ small enough, for any standard pair $\ell$ which intersects $\{\theta=\theta_0\}$
  there exists a set $Q_h\subset\supp\ell$ so that $\mu_\ell(Q_h)>\expo{-\const\vei T}$
  and $F_\ve^{\pint{T\vei}}Q_h\subset\bT^1\times B(\theta_1,\Const\ve^{ 5/12})$.\footnote{ Note that \cite[Theorem 6.3]{DeL2} does not provide the explicit $T$ dependence in the lower bound of the measure of $Q_h$ stated here, however the latter is not needed to prove that the measure is strictly positive. On the other end, once there exists one trajectory in $\supp \ell$ that ends up in $B(\theta_1,\frac 12\Const\ve^{ 5/12})$, then trajectories that start in an $\exp{-\const \ve^{-1} T}$ neighborhood will depart from such a trajectory less than $\frac 12\Const\ve^{ 5/12}$ in time $T$, hence the current claim.}
\end{thm}
In fact we can also obtain a Local Central Limit Theorem:
\begin{thm}[{\cite[Theorem 2.7]{DeL1}}]\label{thm:lclt}
 For any $T>0$, there exists $\ve_0>0$ so that the following holds.  For any $\beta>0$,
    compact interval $I\subset\bR$, $|I|\leq 1$, real numbers $\shiftPar>0$, $\ve\in(0,\ve_0)$,
    $t\in[\ve^{1/2000},T]$, and standard pair $\ell\in \stdpSet{\spc1,\spc2}$, we have:
  \begin{equation} \label{e_indicatorlclt}
 \frac{ \mu_\ell(\deviation\theta(t,\cdot)\in
    \ve I + \shiftPar\veh)}{\sqrt\ve}= \Leb\, I\left[
      \frac{e^{-\shiftPar^2/2\Var_t^2(\thetasl)}}{\Var_t(\thetasl)\sqrt{2\pi}}\right]+\cO(\ve^{\efrac 12-\beta}).
  \end{equation}
  where the variance $\Var_t^2(\theta)$ is given by
  \begin{equation}\label{e_variancelclt}
  \Var_t^2 (\theta)= \int_0^t e^{2\int_s^t\bar\omega'(\bar\theta(r,\theta))\deh r}\bVar^2(\bar\theta(s,\theta))\deh s.
  \end{equation}
 and $\bVar^2:\bT\to\bRp$ is defined in \eqref{e_definitionBarChi}.
\end{thm}
\begin{rem}
Essentially, the above theorem states that the distribution of $\theta_\ve(t)$ looks like it has a regular density up to scale $\ve^{\frac 32-\beta}$. To determine the properties of the distribution below such a scale, it requires further investigation (possibly in the spirit of \cite{AGT}).
\end{rem}
Observe that, $\bVar$ defined above is bounded away from $0$ by assumption
$\ref{a_noCobo}$, hence we conclude that
\begin{align}\label{e_varianceEstimate}
  \Const t \leq \Var^2_t\leq\Const\expo{\const t}t
\end{align}

The above results will be instrumental in the following section.

\section{Deterministic versus random}\label{sec:det-rand}
The goal of this section is to show that the results of the previous section can be used to predict very precisely the behavior of the system in the long time regime.

For each $\theta\in\bT$ let us define the SDE
\begin{equation}\label{e_deviationDiffusion}
\begin{split}
&  \deh\deviation(t,\theta) = \bar\omega'(\bar\theta(t, \theta))\deviation(t)\deh t + \bVar(\bar\theta(t,\theta)\deh B(t)\\
&\deviation(0,\theta)=0,
\end{split}
\end{equation}
where $B(t)$ is a standard Brownian motion.
Note that, for a fixed standard pair $\ell$, the distribution on the left hand side of \eqref{e_indicatorlclt} is exactly the distribution, at time $t$, of the solution of the above SDE for $\theta=\thetasl$.
Indeed, in such a case, the solution of \eqref{e_deviationDiffusion} is given by
\[
\deviation(t,\thetasl) =\int_0^t
e^{\int_s^t\bar\omega'(\bar\theta(\tau,\thetasl))d\tau}\bVar(\bar\theta(s,\thetasl)) \deh
B(s),
\]
which is a zero mean Gaussian random variable with variance given by \eqref{e_variancelclt}.

Given any standard pair $\ell$ we can then define the following Markov process $\Beta_0$: Fix $T\in\bR_+$. For each $\theta\in\bT$ define the auxiliary process
\[
\Beta_*(t, \theta)=\bar\theta(t,\theta)+\sqrt\ve \deviation (t,\theta).
\]
Then, for $t\in [0,T]$,  $\Beta_0$ is defined as
\begin{equation}\label{eq:processzero}
  \Beta_0(t, \thetasl)=\bar\theta(t,\thetasl)+\sqrt\ve \deviation (t,\thetasl).
\end{equation}
While, for times $t\in [kT,(k+1)T]$ the process is defined by the Markov property
\[
\bE(f(\Beta_0(s+kT, \thetasl)))=\bE(f(\Beta_*(s-kT, z))\;|\; \Beta_0(kT,\thetasl)=z).
\]
In the following we will suppress the dependence on the initial point, if it does not cause any confusion.
From Theorem \ref{thm:lclt} it follows:
\begin{lem}\label{cor:first-step}
For any $\beta>0$, $\alpha\in(0,\beta)$, $\ve\in (0,\ve_0)$, standard pair $\ell$ and $t\in[\ve^{1/2000},\ve^{-\alpha}]$, there exists a coupling $\bP_c$ between $\theta_\ve(t)$, under $\mu_\ell$, and $\Beta_0(t,\thetasl)$:
  \begin{equation} \label{e_indicatorlclt-long}
\bP_c(|\theta_\ve(t)-\Beta_0(t,\thetasl)|\geq \ve)=\cO(\ve^{1/2-\beta}).
  \end{equation}
\end{lem}

\begin{proof}
For brevity we call a coupling between two random variables $X,Y$ such that $\bP(|X-Y|\geq \ve)\leq\delta$ an $(\ve,\delta)$-coupling.

The proof is by induction, let us prove it first for $t\in[\ve^{1/2000},T]$.
We partition $\bR$ in bins $I_n$ of size $1$ centered around
$b_n=n$.  By Theorem \ref{thm:lclt} we have for any $\zeta>0$,
\[
\begin{split}
\mu_\ell(\theta_\ve(t,\cdot)\in\bar\theta(t,\thetasl) + \ve I_n) &= \int_{\ve I_n}
    \frac{e^{-y^2/(2\Var_t^2\ve)}}{\Var_t\sqrt{2\pi\ve}} dy+\cO(\ve^{1-\zeta/2})\\
    &=\bP(\Beta_0(t)\in \bar\theta(t,\thetasl) +  \ve I_n)+\cO(\ve^{1-\zeta/2}).
\end{split}
\]
We can thus construct a coupling of the part of measure that belongs to the same bins, and if we do it for all $n\leq \ve^{-1/2-\zeta/2}$, then we have a mass of order $\cO(\ve^{1/2-\zeta})$ that cannot be coupled. On the other hand by Theorem \ref{t_largeDevz} the total mass for both process that can belong to intervals with $n\geq \ve^{-1/2-\zeta/2}$ is also smaller than $\ve^{1/2-\zeta}$ (in fact much smaller). We have constructed an $(\ve, C\ve^{1/2-\zeta})$-coupling for some $C>0$, and the Lemma is thus proven for $t\in[\ve^{1/2000},T]$.

Next, let us assume that each $t\leq kT$ we can construct an $(\ve, 2kC\ve^{1/2-\zeta})$-coupling. In particular, the bound holds for $t\leq T_k= kT-\ve^{1/2000}$. Let $\{I_n\}$ be a partition of $\bT$ in intervals of size $\ve$.
We know that the standard pair $\ell$ will give rise, at time $T_k$, to a standard family $\stdf_k=(\alphaset_k,\cF_k,\fm_k)$ such that
\[
\mu_\ell(\theta_\ve(t,\cdot)\in I_n)=\sum_{\alpha\in\alphaset_k}\fm_\alpha\mu_{\ell_\alpha}(\theta_\ve(t-T_k,\cdot)\in I_n).
\]
Then
\[
\mu_\ell(\theta_\ve(t,\cdot)\in I_n)=\sum_j\sum_{\{\alpha\;:\; \theta_{\ell_\alpha}\in I_j\}}\fm_\alpha\mu_{\ell_\alpha}(\theta_\ve(t-T_k,\cdot)\in I_n).
\]
By the inductive hypothesis we can make an $(\ve, 2kC \ve^{1/2-\zeta})$-coupling with $\Beta_0(T_k)$.
On the other hand it is easy to check that, for each $I\subset \bR$ and $\theta\in\bT^1$,
\[
e^{-\Const \sqrt\ve}\leq\frac{\bP(\Beta_0(t,\theta)\in I)}{\bP(\Beta_0(t,\theta+\ve)\in I)}\leq e^{\Const \sqrt\ve}.
\]
Thus, for each $\theta_{\ell_\alpha}\in I_j$ we can $(\ve, C\ve^{1/2-\zeta})$-couple $\theta_\ve(t)$, under $\mu_{\ell_\alpha}$, with $\Beta_0(t,\theta_{\ell_\alpha})$. This clearly, produces an $(\ve, C(2k+1)\ve^{1/2-\zeta})$-coupling up to time $(k+1)T-\ve^{1/2000}$. Another step as before will allow to construct an $(\ve, 2C(k+1)\ve^{1/2-\zeta})$-coupling up to time $(k+1)T$.

We can iterate this procedure up to $k\leq \ve^{-\alpha}$ for $\alpha<\beta-\zeta$, and get an $(\ve,\cO(\ve^{1/2-\beta}))$-coupling. The Lemma is thus proven, taking $\zeta$ small enough.

\end{proof}

In fact, it is possible to couple our process to the following, more interesting,
Freidlin--Wentzell type equation \eqref{eq:f-w}
\[
\begin{split}
&\deh\Beta(t)=\bar\omega(\Beta(t))\deh t+\sqrt \ve\bVar(\Beta(t))\deh B\\
&\Beta(0)=\thetasl
\end{split}
\]
and $B$ is the standard Brownian motion.

Since $F_\ve\in\cC^5$, we can apply \cite[Theorem 8.1]{GL06} with the Banach spaces $\{\cC^i\}_{i=0}^s$, $s=4$, and obtain that $\bar\omega, \bVar\in\cC^{4-\alpha}$, for all $\alpha>0$. Thus \cite[Theorem
2.2]{WeFr} implies that $\Beta(t)=\Beta_0(t)+\ve\Beta_2(t)+\ve^{\frac 32}\Beta_r(t)$ where
\[
\begin{split}
&d\Beta_2(t)=\bar\omega'(\bar\theta(t))\Beta_2(t) dt +\frac 12 \bar\omega''(\bar\theta(t))\deviation(t)^2 dt+\bVar' (\theta(t))\deviation(t) dB(t)\\
&\Beta_2(0)=0,
\end{split}
\]
and
\[
\bE(\Beta_r(t)^2)\leq \Const .
\]
Hence, for each $\beta>0$,
\[
\bP(|\Beta_r(t)|\geq \ve^{-\beta})\leq \Const \ve^{2\beta}.
\]
%Let let $d$ be a distance on $\bT$ and let us define the Wasserstein distance
%\[
%D_d(X_1,X_2)=\inf_\gamma \,\bE_\gamma (d(X_1,X_2))
%\]
%where $\gamma$ is a coupling of $X_1,X_2$. In the following we will be interested in the discrete distance $d_0(x,y)=0$ iff $x=y$ and its $\ve$ approximation $d_*(x,y)=\ve^{-1+\alpha_0/4}|x-y|$ if $|x-y|\leq \ve^{1-\alpha_0/4}$ and $d_*(x,y)=1$ otherwise.

Let us denote by $\deh_{TV}$ the total variation distance.

\begin{lem}\label{lem:second-step}
For any $\beta>0$, $\alpha\in (0,\beta)$ and $t\in[0,\ve^{-\alpha}]$ we have
\[
\deh_{TV}(\Beta(t),\Beta_0(t))=\cO( \ve^{1/2-\beta}).
\]
\end{lem}
\begin{proof}
Again we start by considering the time interval $[0,T]$ first.
It should be possible to prove the Lemma by using the above estimates for $\Beta_2,\Beta_r$, yet, we find faster to use the following result from \cite{Kifer76}:\footnote{ Related results are present in \cite{FS86}, where they are investigated from the point of view of viscosity solutions.} for each $t\in [0,T]$, let $p^\ve_t(\thetasl,\bar\theta(t,\thetasl)+y)$ be the distribution of the random variable $\Beta(t)$ determined by \eqref{eq:f-w}, then
\begin{equation}\label{eq:approx transition}
\left|p^\ve_t(\thetasl,\bar\theta(t,\thetasl)+y)-(2\pi\ve)^{-\frac 12}e^{-V(t,y)/2\ve}K_0(t,y)\right|\leq \Const\ve^{\frac12} e^{-V(t,y)/2\ve}
\end{equation}
where, setting $D_{y,t}=\{\vf\in\cC^1\st \vf(0)=\thetasl,\vf(t)=\bar\theta(t,\thetasl)+y\}$,
\begin{equation}\label{eq:infimumV}
V(t,y)=\inf_{\vf\in D_{y,t}}\int_0^t \frac{(\dot\vf(s)-\bar\omega(\vf(s))^2}{\bVar(\vf(s))^2}\deh s
\end{equation}
and $K_0\in\cC^1$. To compute $V$ note that the minimum is attained on the solution $\vf$ of the Euler-Lagrange equations.

To the Lagrangian is associated the Hamiltonian $\cH(\vf, p)=\frac{\bVar(\vf)^2}4p^2+p\,\bar\omega(\vf)$ where $p=2\bVar(\vf)^{-2}(\dot\vf-\bar\omega(\vf))$. The Hamiltonian equation of motions read
\begin{equation}\label{eq:hamilton}
\begin{split}
&\dot\vf=\partial_p \cH=\frac{\bVar(\vf)^2}2 p+\bar\omega(\vf)\\
&\dot p=-\partial_\vf\cH=-\frac{\bVar(\vf)\bVar'(\vf)}2p^2-\bar\omega'(\vf) p.
\end{split}
\end{equation}
Note that $(\vf(s), p(s))=(\bar\theta(s,\thetasl),0)$ is a solution of \eqref{eq:hamilton} with initial condition $\thetasl$ and final condition $\bar\theta(t,\thetasl)$, which corresponds to $y=0$.

We can linearise the equation of motion around the solution $(\bar\theta(s,\thetasl),0)$. Let $(\vf(t, p_0),p(t,p_0))$ be the solution of \eqref{eq:hamilton} with  initial conditions $(\thetasl,p_0)$ and set $(\xi(t),\eta(t))=\partial_{p_0}(\vf(t, p_0),p(t,p_0)|_{p_0=0}$. Then
\[
\begin{split}
&\dot\xi=\frac{\bVar(\theta)^2}2 \eta+\bar\omega'(\bar\theta)\xi\\
&\dot \eta=-\bar\omega'(\bar\theta) \eta\\
&\xi(0)=0\,,\quad \eta(0)=1,
\end{split}
\]
It readily follows
\[
\begin{split}
&\eta(s)=e^{-\int_0^s\bar\omega'(\bar\theta(s_1),\thetasl)ds_1}\\
&\xi(s)=e^{\int_0^s\bar\omega'(\bar\theta(s_1),\thetasl)ds_1}\int_0^{s} e^{-2\int_0^{s_1}\bar\omega'(\bar\theta(s_2),\thetasl)ds_2}
\frac{\bVar(\bar\theta(s_1,\thetasl))^2}2 ds_1\\
&\phantom{\xi(s)}=\frac{\Var_s(\thetasl)^2}2
e^{-\int_0^s\bar\omega'(\bar\theta(s_1),\thetasl)ds_1}.
\end{split}
\]
Also note that, since $\bar\omega\in\cC^3$ (see discussion after \eqref{eq:averageeq}) also $p(t,\cdot)\in\cC^3$.
If we want the solution belonging to $D_{y,t}$, then we have to solve the equation $F(p_0,y)=\vf(t,p_0)-\bar\theta(t,\thetasl)-y=0$. Since $\partial_{p_0}F=\xi(t)\neq 0$, we can apply the implicit function theorem in a neighborhood of $(0,0)$ and conclude that the minimum in \eqref{eq:infimumV} is obtained for the solution of \eqref{eq:hamilton} with initial conditions
\[
\vf(0)=\thetasl\;;\quad p_0(y)=\xi(t)^{-1}y+\cO(y^2),
\]
where $p_0\in\cC^3$.
Moreover, since we have
\[
V(t,y)=\int_0^t \frac{p(s,p_0(y))^2\bVar(\vf(s, p_0(y))))^2}{2}\deh s,
\]
the above shows that $V(t,\cdot)\in\cC^3$ in a neighborhood of zero and allows to compute
\[
\begin{split}
&V(t,0)=0\;;\quad \partial_y V(t,0)=0\\
&\partial_y^2V(t,0)=\int_0^t\eta(s)^2\xi(t)^{-2}\bVar(\bar\theta(s,\thetasl))^2=2 \Var_t(\thetasl)^{-2},
\end{split}
\]
which, by Taylor expansion, yields
\begin{equation}\label{eq:approx V}
V(t,y)=\frac{y^2}{\Var_t^2}+\cO(y^3).
\end{equation}
This, together with the large deviation results in \cite{WeFr}, shows in particular that
\[
\int_{|y|\geq \ve^{1/2-\zeta}}p^\ve_t(\thetasl,\bar\theta(t,\thetasl)+y)\leq \ve^{100}.
\]
It suffices thus to consider $|y|\leq \ve^{1/2-\zeta}$, for which we have
\begin{equation}\label{eq:approxtransition2}
\left|p^\ve_t(\thetasl,\bar\theta(t,\thetasl)+y)-(2\pi\ve\Var_t^2)^{-\frac 12}e^{-\frac{y^2}{2\ve\Var_t^2}}\right|\leq \Const\ve^{1/2-3\zeta} e^{-\frac{y^2}{2\ve\Var_t^2}}.
\end{equation}
This proves the statement for $t\leq T$. To prove the result for longer times, we proceed by recurrence, relying on Markov Property: after each time interval of length $T$ there is an additional amount of mass of order $\cO(\ve^{1/2-3\zeta})$ that cannot be coupled, which means that $\deh_{TV}(\Beta(t),\Beta_0(t))=\cO( k\ve^{1/2-3\zeta})$ for $t\in[kT,(k+1)T]$.
We deduce the result, taking $\alpha\leq \beta-3\zeta$ and $\zeta$ small.

\end{proof}
By Lemmata \ref{cor:first-step} and \ref{lem:second-step} immediately follow
\begin{cor}\label{cor:third-step}
For any $\beta>0$, $\alpha\in(0,\beta)$, $\ve\in (0,\ve_0)$, standard pair $\ell$ and $t\in[\ve^{1/2000},\ve^{-\alpha}]$, there exists a coupling $\bP_c$ between $\theta_\ve(t)$, under $\mu_\ell$, and $\Beta(t)$, such that:
\[
\bP_c(|\theta_\ve(t)-\Beta(t)|\geq \ve)=\cO(\ve^{1/2-\beta}).
\]
\end{cor}

\section{Long times}\label{sec:f-w}

We have thus seen that our deterministic process remains very close, in a precise technical sense, to the Freidlin-Wentzell process \eqref{eq:f-w} for a polynomially long time. To take advantage of this it is necessary to have a good understanding of the statistical properties of \eqref{eq:f-w}. To this issue is devoted the present section.

\medskip

First of all, note that the generator associated to the process can be written as
\begin{equation}\label{eq:generator}
\begin{split}
&L_\ve\vf=\bar\omega\vf'+\frac \ve 2\bVar^2\vf''=\frac{\ve}{2 \rho_\ve}(\bVar^2\rho_\ve\vf'-Z_\ve\drifte \vf)',\\
&\rho_\ve(\theta)=Z_\ve\bVar^{-2}e^{-\vei\Omega(\theta)}\left[1+\drifte\int_{0}^\theta e^{\vei\Omega(s)} ds \right],\\
&\Omega(\theta)=-2\int_{0}^\theta\frac{\bar\omega(s)}{\bVar^2(s)}ds\,\textrm{ for all }\theta\in \bR,
\end{split}
\end{equation}
where $\drifte\in\bR$ is determined by the relation $\rho_\ve(0)=\lim_{\theta\to1}\rho_{\ve}(\theta)$, which insures that $\rho_\ve$ is a smooth periodic function and hence the measure factors properly on $\bT$, while the normalisation constant $Z_\ve$ is determined by $\int_{\bT}\rho_\ve=1$.
Accordingly,
\begin{equation}\label{eq:drifte}
\drifte=\frac{e^{\vei\Omega(1)}-1}{\int_{0}^{1}e^{\vei\Omega(s)}ds},
\end{equation}
which shows that $\drifte=0$ if and only if $\int_{\bT}\frac{\bar\omega}{\bVar^2}=0$.
Also, note that,
\[
\begin{split}
1+\drifte\int_{0}^\theta e^{\vei\Omega(s)} ds &=\frac {\int_{0}^{1}e^{\vei\Omega(s)}ds+\left[e^{\vei\Omega(1)}-1\right]\int_{0}^\theta e^{\vei\Omega(s)} ds }{\int_{0}^{1}e^{\vei\Omega(s)}ds}\\
&=\frac {\int_{\theta}^{\theta+1}e^{\vei\Omega(s)}ds }{\int_{0}^{1}e^{\vei\Omega(s)}ds},
\end{split}
\]
and thus $\rho_\ve$ can also be written as
\begin{equation}\label{eq:integral form}
\begin{split}
\rho_\ve(\theta)&= \frac{\int_\theta^{\theta+1}\bVar^{-2}(\theta)e^{-\ve^{-1}(\Omega(\theta)-\Omega(s))} ds}
{ \int_0^1\int_\theta^{\theta+1}\bVar^{-2}(\theta)e^{-\ve^{-1}(\Omega(\theta)-\Omega(s))} ds\,d\theta} \\&
=
\tilde Z_\ve
\int_\theta^{\theta+1}\bVar^{-2}(\theta)e^{-\ve^{-1}(\Omega(\theta)-\Omega(s))} ds
.
\end{split}
\end{equation}
One can easily check that $L_\ve'\rho_\ve=0$. That is, $\rho_\ve$  is the density of the invariant measure $\nu_\ve$
of \eqref{eq:f-w}. In addition, if $\drifte=0$, then the process is reversible.

\medskip

Let us first consider the case when $\bar \omega$ has no zeros. We suppose $\bar \omega>0$ (the negative case is obtained by symmetry). Then for a constant $c$ big enough, since $\Omega(\theta)-\Omega(s)\geq \const (s-\theta)$ for $s>\theta$,
\[
\begin{split}
\int_\theta^{\theta+1}\bVar^{-2}(\theta)e^{-\ve^{-1}(\Omega(\theta)-\Omega(s))} ds&=\int_\theta^{\theta+c\ve|\log \ve|}\bVar(\theta)^{-2}e^{-\ve^{-1}(\Omega(\theta)-\Omega(s))} ds+O(\ve^2)\\
&=\bVar(\theta)^{-2}\int_0^{c\ve|\log \ve|}e^{\ve^{-1}\Omega'(\theta)s}d s+O(\ve^2|\log \ve|^2)\\
&=-\ve \bVar(\theta)^{-2}(\Omega')(\theta)^{-1}+O(\ve^2|\log \ve|^2)\\
&= \frac{\ve}{2} \bar \omega(\theta)^{-1} +O(\ve^2|\log \ve|^2).
\end{split}
\]
We deduce, after normalization, that $\rho_\ve(\theta)=Z\bar \omega(\theta)^{-1}+O(\ve|\log\ve|^2)$, for some $Z\in\bR_+$. We have the following result of convergence toward $\nu_\ve$.

\begin{lem}\label{lem:conver-no-zero} Suppose that $\bar \omega$ has no zeros. There exists $\const$ and $\Const$ such that for all $\theta \in \bT$,
$$
\deh_{TV}( p_t^\ve(\theta,\cdot),\nu_\ve(\cdot))\leq \Const e^{-\const \ve t}.
$$
\end{lem}

\begin{proof}
By standard coupling arguments, it is sufficient to prove that two solutions of \eqref{eq:f-w} with respect to two independent Brownian motions and starting from $x$ and $x'$ meet before a time of order $\ve^{-1}$ with a probability bounded from zero independently from $\ve$ and uniformly in $x$ and $x'$. We follow here arguments developed in \cite{GPS} in a more general context.
We denote by $h$ the isochron map associated to the periodic solution $\bar \theta (t,\theta_0)$ for a $\theta_0\in \mathbb{T}^1$, that is, denote $T$ its period, the mapping from $\mathbb{T}^1$ to $\bR /T \bZ$ satisfying $h(\theta_0)=0$ and $h'(\theta)=\bar \omega^{-1}(\theta)$. We have in particular $h( \bar \theta (t,\theta_0))=t\, \text{mod}\,  T$.

We study then the process $(\Psi_t)_{t\geq 0}$ defined as the lift of $h(\Beta_t)$ (i.e. the unique $\bR$-valued trajectory satisfying $\Psi_0\in [0,T)$ and $h(\Beta_t)=\Psi_t\text{ mod }T$ for all $t\geq 0$). $\Psi_t$ satisfies
\begin{equation}\label{eq:eds Psi}
\deh \Psi_t=\deh t - \ve \frac{\bar \omega' \bVar^2}{2 \bar \omega^2}\circ h^{-1}(\Psi_t)\deh t +\sqrt{\ve}\frac{\bVar}{\bar \omega}\circ h^{-1}(\Psi_t)\deh B_t,
\end{equation}
where, for $u\in \bR$, $h^{-1}(u)$ is to be understood as $h^{-1}(u\text{ mod }T)$.
Now if $T$ denotes the period of the deterministic dynamics $\bar \theta$ defined by \eqref{eq:averageeq}, we have with probability converging to $1$ when $\ve$ goes to $0$ that for $\zeta>0$ small,
\[\max_{n\leq \ve^{-1}} \sup_{t\in [nT,(n+1)T]}|\Psi_t-(\Psi_{nT}+t-nT)|\leq \ve^{\frac12-\zeta}.
\]
Indeed from the Burkholder-Davis-Gundy inequality we get for $m\geq 1$
\begin{multline*}
\bP\left(\sup_{t\in[nT,(n+1)T]}\left| \int_{nT}^{t}\frac{\bVar}{\bar \omega}\circ h^{-1}(\Psi_t)\deh B_t\right|\geq \frac12 \ve^{-\zeta}\right)
\\\leq 2^m\ve^{m\zeta}\bE\left[\sup_{t\in[nT,(n+1)T]}\left| \int_{nT}^{(n+1)T}\frac{\bVar}{\bar \omega}\circ h^{-1}(\Psi_t)\deh B_t\right|^m\right]
\leq C_m \ve^{m\zeta},
\end{multline*}
and we can simply choose $m\zeta >1$. So with probability converging to one at each step of size $T$ the third term in \eqref{eq:eds Psi} gives a contribution of order $\frac12 \ve^{\frac12-\zeta}$, and the second term gives an even smaller contribution (of order $\ve$).

Remark that the functions $u\mapsto \frac{\bar \omega' \bVar^2}{2 \bar \omega^2}\circ h^{-1}(u)$ and $u\mapsto \frac{\bVar}{\bar \omega}\circ h^{-1}(u)$ are $T$ periodic. We denote $v=\int_0^T \frac{\bar \omega' \bVar^2}{2 \bar \omega^2}\circ h^{-1}(u)du$ and $\kappa = \int_0^T \left(\frac{\bVar}{\bar \omega}\circ h^{-1}\right)^2(u)du$. From the above estimates, we deduce that
\begin{equation}\label{eq:estim Psi}
\Psi_{nT}-nT-\Psi_0= \ve n v +\ve^{\frac12}G_{nT}+\ve\int_0^{nT}b^\ve_t\deh t+\sqrt \ve \int_0^{nT}\gamma_t^\ve\deh B_t,
\end{equation}
where $G_{nT}$ is a random variable of normal distribution, centered and with variance $n \kappa$, and for $kT\leq t\leq (k+1)T$ with $k\geq n-1$,
\[
\begin{split}
&b_t^\ve=\frac{\bar \omega' \bVar^2}{2 \bar \omega^2}\circ h^{-1}(\Psi_t)-\frac{\bar \omega' \bVar^2}{2 \bar \omega^2}\circ h^{-1}(\Psi_{kT}+t)=\cO(\ve^{\frac12-\zeta}),\\
& \gamma_t^\ve=\frac{\bVar}{\bar \omega}\circ h^{-1}(\Psi_t)-\frac{\bVar}{\bar \omega}\circ h^{-1}(\Psi_{kT}+t)=\cO(\ve^{\frac12-\zeta}).
\end{split}
\]
Using similar estimates as above, we can prove that the two last terms of \eqref{eq:estim Psi} are of order $\cO(\ve^{\frac12-\zeta})$ with probability converging to $1$, and thus $\Psi_{\lfloor\ve^{-1}\rfloor T}-\lfloor \ve^{-1} \rfloor T-\Psi_0$ converges in distribution to a Gaussian with mean $v$ and variance $\kappa$. This implies indeed that two independent solutions of \eqref{eq:f-w} meet before $\ve^{-1}T$ with a positive probability independent from $\ve$.

\end{proof}

The above result shows that the convergence to equilibrium takes place on a rather long time scale. As we will see in Section \ref{sec:rotation}, with the available technology, this allows only a partial understanding of the properties of the physical measures.

Next we consider the case when $\bar \omega$ admits $2\zeroes$ non-degenerates zeroes, $\zeroes>0$. In such a case the true convergence to equilibrium takes place to an even longer time scale, yet the convergence to a metastable situation takes place much faster.

Let us denote $\theta_{i,-}$ the stable zeroes, for $i=1,\ldots,\zeroes$, and $\theta_{i,+}$ the unstable ones.
We aim at proving the following Proposition.

\begin{prop}\label{prop: close to gaussians}
Suppose that the dynamics defined by \eqref{eq:averageeq} admits $2\zeroes$ non-degenerate fixed points. Then there exists a constant $\const$ such that for all $x$ and all $\alpha>0$ there exists non-negative real numbers $c_1(x),\ldots,c_\zeroes(x)$ satisfying $c_1(x)+\ldots+ c_\zeroes(x)=1$ such that
\begin{equation}
\sup_{t\in [\const |\log \ve|,\ve^{-\alpha}]}\deh_{TV}\left( p_t^\ve(x,\cdot),\sum_{i=1}^\zeroes c_i(x)\cG_i^{\ve}(\cdot)\right) =\cO\left( \ve|\log\ve|^{\frac32}\right),
\end{equation}
where $\cG_i^{\ve}$ is a gaussian distribution with mean $\theta_{i,-}$ and variance $\frac{ \bVar^2(\theta_{i,-})}{2 \bar\omega'(\theta_{i,-})}\ve$.
\end{prop}
\begin{rem} Note that, when starting from a standard pair, the initial conditions can belong, at most, to two sinks, hence all the $c_i$ are zero a part from, at most, two. However, if one considers an initial condition described by a standard family, then all the $c_i$ can be strictly positive.
\end{rem}
We will not investigate the optimal constant $\const$ given in this Proposition, as we are interested in longer times. For results in this direction and related to the cut-off phenomena, see \cite{BarJar}.

For $a>0$ denote $I^a_i=[\theta_{i,-}-a,\theta_{i,-}+a]$. Since $\bar \omega$ is smooth, there exists a $a>0$ such that, on each interval $I_i$, $\bar \omega$ is uniformly convex, with $\bar \omega''(\theta) \geq \const$ for $\theta\in I_i$. We denote $I^a=\cup_{i=1}^n I^a_i$, and first give the following Lemma, which is a particular case of the more general result given in \cite{Kifer81}. For the reader convenience we provide a short proof of this Lemma for our situation.

\begin{lem}\label{lem:trapped}
For all $\beta>0$, there exist constants $\const$ and $\Const$ such that
\begin{equation}
\sup_{t\in [\const |\log \ve|,\ve^{-\alpha}]}\bP\left(\Beta(t)\notin I^a\right)\leq \Const \ve^\beta.
\end{equation}
\end{lem}

\begin{proof}
We divide our analysis in three cases, depending on the position of the starting point $\Beta(0)$.

Let us denote, for any $c>0$ and $i\in\{1,\ldots,n\}$, $J^c_i=\{x\in \bT, |x-\theta_{i,+}|\leq c \}$, $J^c=\cup_{i=1,\ldots,n}J_i^c$ and $\bar J^c$ the complementary of $J^c$.
From large deviations estimates (see \cite{WeFr}), we know that for any $b>0$, with probability $O(e^{-\const \vei})$, the process $(\Beta(t))_{t\geq 0}$ starting from any $x\in \bar J^b$ reaches $I^{\frac{a}{2}}$ before a time $T_b$ independent from $\ve$, and then does not leave $I^a$ before $t=\ve^{-\alpha}$.

Suppose now that $(\Beta(t))_{t\geq 0}$ starts from a point $x \in J^b_i\cap \bar J^{\ve^{1/2-\zeta}}$ for some $\zeta>0$. Then for any $\delta >0$, the solution $\bar \theta(\cdot,x)$ of \eqref{eq:averageeq} reaches $\bar J^b$ before $T_\ve^\delta=\frac{\frac12-\zeta}{\bar \omega'(\theta_{i,+})-\delta}|\log \ve| $ if $b$ is taken small enough. Now we have
\[
\Beta(t)-\bar\theta(t,x)=\int_0^t\big(\bar\omega(\Beta(s))-\bar\omega(\theta(s,x))\big)ds
+\sqrt \ve \int_0^t \bVar(\Beta(s))d B_s,
\]
and from Doob inequality and Burholder-Davis-Gundy inequality, denoting $Z^\ve_t=\int_0^t \bVar(\Beta(s))d B_s$, we get for any $m\geq 1$:
\[
\bP\left(\sup_{t\in[0, T_\ve^\delta] }\left|Z_t^\ve\right|> \ve^{-\xi}\right)\leq \ve^{m\xi}\bE\left[\sup_{t\in[0, T_\ve^\delta] }\big|Z_{T_\ve^b}^\ve\big|^m\right]\leq \Const \ve^{m\zeta} |\log \ve|.
\]
Denote $t_*=\inf\{t\in [0, T_\ve^\delta], |\Beta(t)-\bar\theta(t,x)|\geq b/2\}$. On the event $A^\ve=\{\sup_{t\in[0, T_\ve^\delta] }\left|Z_t^\ve\right|\leq \ve^{-\xi}\}$ and if $b$ is small enough we obtain
\[
|\Beta(t)-\bar\theta(t,x)|\leq \big(\bar\omega'(\theta_{i,+})+\delta\big)\int_0^t|\Beta(s)-\bar\theta(s,x)|ds
+\ve^{\frac12-\xi},
\]
and thus by Gr\"onwall inequality as long as $t\leq t_*$ we have the following upper bound:
\[
|\Beta(t)-\bar\theta(t,x)|\leq \ve^{\frac12-\xi-\frac{\bar \omega'(\theta_{i,+})+\delta}{\bar \omega'(\theta_{i,+})-\delta}\left(\frac12-\zeta\right)}.
\]
If we choose $\xi$ and $\delta$ sufficiently small with respect to $\zeta$, this right hand side term goes to $0$ when $\ve$ tends to $0$, which means that on the event $A^\ve$ (which satisfies $1-\bP(A^\ve)=O(\ve^\beta)$ for any $\beta>0$), if $\ve$ is small enough, $t_*=T_\ve^\delta$ and thus $(\Beta(t))_{t\geq 0}$ reaches $\bar J^{b/2}$ before $t=T_\ve^\delta$.

We suppose now that $(\Beta(t))_{t\geq 0}$ starts from a point $x\in J_i^{\ve^{1/2-\zeta}}$, and consider the process $(y(t))_{t\geq 0}$ starting from $x$ and satisfying
\[
\deh y(t)=\bar \omega'(\theta_{i,+})\big(y(t)-\theta_{i,+}\big)\deh t+\sqrt \ve \bVar(\theta_{i,+})\deh B_t.
\]
$y(t)-\theta_{i,+}$ has in fact a Gaussian distribution (projected on the torus) of mean $e^{\bar \omega'(\theta_{i,+})t}(x-\theta_{i,+})$ and variance $\ve \bVar^2(\theta_{i,+})\left(\frac{e^{2\bar\omega(\theta_{i,+})t}-1}{2\bar\omega(\theta_{i,+})} \right)$. So considering a time $t_\ve=\frac{\gamma}{\bar \omega'(\theta_{i,+})}|\log\ve|$ for a $\gamma>0$, we have $\bP(y(t_\ve)\in J_i^{2\ve^{1/2-\zeta}})=O(\ve^{\gamma-\zeta})$. On the other hand, comparing the two processes, we obtain
\[
\deh \big(\Beta(t)-y(t)\big)=\bar \omega'(\theta_{i,+})\big(\Beta(t)-y(t)\big)\deh t
+h_1(\Beta(t)-\theta_{i,+})\deh t +h_2(\Beta(t)-\theta_{i,+})\deh B_t,
\]
where $h_1(u)=O(u^2)$ and $h_2(u)=O(u)$. So if we denote $\tau^\ve$ the exit time from $J^{\ve^{1/2-\zeta}}$ for the process $(\Beta(t))_{t\geq 0}$, then
\[
\bP\left(\sup_{t\in[0,t_\ve]}\left| \int_{0}^{t\wedge \tau^\ve}h_2(\Beta(t)-\theta_{i,+})\deh B_t\right|\geq \ve^{1/2-2\zeta}\right)\leq \Const\ve^{m\zeta}|\log\ve|,,
\]
which implies that for $t\in [0,t_\ve]$
\[
|\Beta(t\wedge\tau^\ve)-y(t\wedge\tau^\ve)|\leq \Const \ve^{1-2\zeta-\gamma}.
\]
This proves, for $\zeta$ and $\gamma$ small enough, that $(\Beta(t))_{t\geq 0}$ reaches $\bar J^{\ve^{1/2-\zeta}}$ before $t_\ve$ with a probability $1-O(\ve^{\gamma-\zeta})$.

We have thus proved, considering these different cases, that after a time of order $c |\log\ve|$ the process $(\Beta(t))_{t\geq 0}$ is trapped in $I^a$ until $t=\ve^{-\alpha}$ with a probability of order $1-O(\ve^\iota)$ for some $\iota>0$. Proceeding by recurrence, considering the trajectories that are not trapped yet, we prove that the process is in $I^a$ with probability $1-O(\ve^\beta)$ for any $\beta>0$, taking the constant $c$ large enough.

\end{proof}

Lemma \ref{lem:trapped} shows that after a time of order $|\log\ve|$ (that allows the process to escape from the neighbourhoods of the unstable fixed points), the process stays with high probability in one of the intervals $I_i$. We can thus, when $(\Beta(t))_{t\geq 0}$ is trapped in $I_i$, couple $(\Beta_t-\theta_{i,-})_{t\geq 0}$ with the process $(x(t))_{t\geq 0}$ defined on the real line by the equation
\begin{equation}\label{eq:sde f}
\deh x(t)= -f_i'(x(t)) \deh t + \sqrt{\ve} g_i(x(t)) \deh B_t,
\end{equation}
with $f_i$ being smooth, satisfying $f_i(0)=0$, $\const\leq f_i''\leq \Const$ and such that $-f'_i(x)=\bar \omega (\theta_{i,-}+x)$ for $|x|\leq a$, and $g_i$ being smooth, $\const\leq g_i\leq \Const$ and such that $g_i(x)=\bVar(\theta_{i,-}+x)$ for $|x|\leq a$.

Remark that the process $(x(t))_{t\geq 0}$ admits the invariant measure $\pi^\ve_i$ with density $h^\ve_i(x)=Z_i^\ve g_i^{-2}(x)e^{-\vei W_i(x)} $, where $W_i(x)=2\int_{0}^x\frac{f'_i(s)}{g_i^2(s)}ds$ and $Z_i^\ve$ is a normalization constant. Then, for any $\beta$, we get for $c$ large enough,
\[
\begin{split}
\int_\bR (Z^\ve_i)^{-1} h^\ve_i(x)d x&=\int_{-c|\log\ve|^{\frac12}}^{c|\log\ve|^{\frac12}}(Z^\ve_i)^{-1} h^\ve_i(x)d x+O(\ve^{\beta})
\\&=\int_{-c|\log\ve|^{\frac12}}^{c|\log\ve|^{\frac12}}g^{-2}(0)e^{-\vei W_i''(0)\frac{x^2}{2}}(1+O(\ve^{\frac12}|\log\ve|^{\frac32}))+O(\ve^\beta),
\end{split}
\]
so $Z_i^\ve=\left(\frac{W''_i(0)}{2\pi \ve} \right)^{-\frac12}+O(\ve|\log\ve|^{\frac32})$, and a similar calculation shows that $\deh_{TV}(\pi_i^\ve,\cG_i^{\ve})=O(\ve|\log\ve|^{\frac32})$.

To conclude the proof of Proposition \ref{prop: close to gaussians}, it suffices now to prove that $x_t$ converges in distribution to $\pi_i^\ve$ fast enough, as stated in the following Lemma. We denote by $P^i_t$ the transition probabilities associated to the process $(x(t))_{t\geq 0}$.

\begin{lem}\label{lem:conv toward pi}
There exist positive constants $\const$ and $\Const$ such that for all $x\in[-a,a]$,
\begin{equation}
\deh_{TV}(  P^i_t(x,\cdot),\pi_i^\ve(\cdot)) \leq \frac{\Const}{\ve}e^{-\const t};
\end{equation}
\end{lem}

To prove this Lemma, we rely on the Harris recurrence type result of M. Hairer and J. Mattingly \cite{HM}, that we recall in the following Theorem.

\begin{thm}[Hairer, Mattingly]\label{th:HM}
Consider a Markov kernel $\cP$ satisfying $\cP V(x)\leq \gamma  V(x)+K$ where $V$ is a Lyapunov function, and $K\geq 0$, $\gamma\in (0,1)$ are constants such that $\inf_{x\in \cC}\cP(x,\cdot)\geq  \zeta \nu(\cdot)$ for a constant $ \zeta \in (0,1)$, a probability measure $\nu$, and the set $\cC=\{V(x)\leq R\}$ with $R>\frac{2K}{1-\gamma}$.
Then we have
\begin{equation}
\rho_\beta(\cP \mu_1,\cP\mu_2)\, \leq \, \bar  \zeta  \rho_\beta(\mu_1,\mu_2)\, ,
\end{equation}
where $\rho_\beta$ is the weighted variational distance
\begin{equation}
\rho_\beta(\mu_1,\mu_2)\, =\, \int_X (1+\beta V(x))|\mu_1-\mu_2|(dx)\, ,
\end{equation}
and one can choose for any $ \zeta _0\in(0, \zeta )$ and $\gamma_0\in (\gamma+\frac{2K}{R},1)$,
\begin{equation}
\beta\, =\, \frac{ \zeta_0}{K}\, ,
\end{equation}
and
\begin{equation}
\bar  \zeta  \, =\, (1-( \zeta - \zeta_0))\vee \frac{2+R\beta\gamma_0}{2+R\beta}\, .
\end{equation}
\end{thm}

\begin{proof}[Proof of Lemma \ref{lem:conv toward pi}.]
Denoting $L_i$ the diffusion operator associated to \eqref{eq:sde f}, we get
\begin{equation}
L_i f_i = -(f_i')^2+\frac{\ve g_i^2}{2}f_i''.
\end{equation}
Since $f_i''\geq \const$, then $(f_i')^2\geq 2 \const f$, which leads to, recalling that both $f_i''$ and $g_i$ are bounded,
\begin{equation}
L_if_i \leq-\const f_i+\Const\ve.
\end{equation}
We deduce the following inequality for the kernel:
\begin{equation}
P^i_t f_i \leq e^{-\const t}f_i+\Const \ve\, .
\end{equation}
We can then consider the dynamics at integer times and we denote $\cP=P^i_1$. It follows
\begin{equation}
\cP f_i \leq \gamma f_i +K\ve\, ,
\end{equation}
where $0<\gamma<1$, and $\gamma$ and $K$ do not depend on $\ve$.
Moreover we can rewrite \eqref{eq:approx transition} but this time for the probability kernel $\cP(x,\cdot)$ of the process $(x_t)_{t\geq 0}$, relying again on \cite{Kifer76}, and make the same estimates as in the proof of Lemma \ref{lem:second-step} to get a bound of the type \eqref{eq:approxtransition2} for $\cP(x,\cdot)$, which means that we can find for any $b>0$ a constant $\zeta>0$ that does not depend on $\ve$ and a probability measure $\mu^\ve$ such that
\begin{equation}
\inf_{|x|\leq b\sqrt{\ve}} \cP(x,\cdot)\geq \zeta \mu^{\ve}.
\end{equation}
Since the set $\cC=\{x\;:\;-f_i'(x)\leq R\}$, where $R$ satisfies $R>\frac{2K\ve}{1-\gamma}$, is included in $\{x\;:\; |x|\leq b\sqrt{\ve}\}$ when $b$ is large enough, our system satisfies the conditions to apply Theorem \ref{th:HM}: we can find constants $\bar \zeta$ and $\beta$ that do not depend on $\ve$ such that
\begin{equation}
\rho_{\beta/\ve}(P^i_{n}(x,\cdot),\pi^\ve_i(\cdot)) \leq \bar \zeta^n \rho_{\beta/\ve}(\delta_x,\pi_i^\ve),
\end{equation}
which implies (remark that $\int f_i\deh \pi_i^\ve \leq \Const \sqrt\ve$) that for $x\in [-a,a]$,
\begin{equation}
\deh_{TV}(P^i_{n}(x,\cdot),\pi_i^\ve(\cdot)) \,\leq \, \frac{\Const}{\ve}\bar \zeta^n\, .
\end{equation}
The result for non-integers times follows directly, since  $\deh_{TV}( P^i_{t+\delta}(x,\cdot),\pi_i^\ve(\cdot))\leq \deh_{TV}( P^i_{t}(x,\cdot),\pi_i^\ve(\cdot))$ for $\delta>0$.
\end{proof}

\section{On the structure of the SRB measures}
First of all let us recall that the work of Tsujii \cite{Tsujii-acta} implies that for partially hyperbolic endomorphisms, generically, the physical  measures are finitely many and absolutely continuous with respect to Lebesgue, this applies to maps of the form \eqref{eq:f-eps}. Unfortunately, on the one hand it is not clear how to check if such a property holds for a specific system, on the other hand such result provides little information of how the physical measure looks like. On the contrary, quite a bit of informations can be obtained by the results established in the previous sections.

Note that, for each standard family $\stdf$, the set of averages of the pushforward $\frac 1n\sum_{k=0}^{n-1}(F_\ve)_*\mu_\stdf$ is weakly compact, hence it has accumulation points. Clearly such accumulation points are invariant measures. Let $\cM_{\operatorname{sp}}(F_\ve)$ be the closure of the set of such accumulation points, thus $\cM_{\operatorname{sp}}(F_\ve)$ is a subset of the invariant measures of the system. Such a class of measure is often called U-Gibbs. See \cite{dolgonote} for a presentation of their properties that far exceeds our present needs.

Recall that any physical measure $\mu$ must belong to $\cM_{\operatorname{sp}}(F_\ve)$ (see \cite[Lemma 9.8]{DeL2}). It is then natural to study the set $\cM_{\operatorname{sp}}(F_\ve)$. This is a set well behaved with respect to the ergodic properties as the next Lemma shows.
\begin{lem}\label{lem:ergodicity} If $\mu\in \cM_{\operatorname{sp}}(F_\ve)$, then its ergodic decomposition consists of measures that also belong to $\cM_{\operatorname{sp}}(F_\ve)$.
\end{lem}
\begin{proof}
Note that, by definition, $\cM_{\operatorname{sp}}(F_\ve)$ is a convex compact set. Hence, by Krein-Milman theorem, each measure in  $\cM_{\operatorname{sp}}(F_\ve)$ can be seen as the convex combination of its extremal points. On the other hand, if $\mu\in \cM_{\operatorname{sp}}(F_\ve)$ and $\tilde\mu$ is invariant and absolutely continuous with respect to $\mu$, then $\tilde\mu\in \cM_{\operatorname{sp}}(F_\ve)$. Indeed, let $h$ be the Radon-Nikodym derivative of $\tilde\mu$ with respect to $\mu$, then, for each $\vf\in\cC^0$ and $n\in\bN$,
\[
\tilde\mu(\vf)=\tilde\mu(\vf\circ F^n_\ve)=\mu(h\cdot\vf\circ F^n_\ve).
\]
By standard approximation arguments, for each $\delta>0$ there exists $h_{\delta}\in\cC^\infty$ such that $\|h-h_{\delta}\|_{L^1(\mu)}\leq \delta$, hence
\[
\tilde\mu(\vf)=\mu(h_{\delta}\cdot\vf\circ F^n_\ve)+\cO(\delta\|\vf\|_{\cC^0}).
\]
In addition, by definition, for each $\boldsymbol\epsilon>0$ there exists a standard family such that
\begin{equation}\label{eq:mu-fam-app}
\mu(\vf)=\sum_{\alpha\in\alphaset}\fm_\alpha\mu_{\ell_\alpha}(\vf)+\cO(\boldsymbol\epsilon \|\vf\|_{\cC^0}).
\end{equation}
Accordingly,
\[
\tilde\mu(\vf)=\sum_{\alpha\in\alphaset}\fm_\alpha\frac 1n\sum_{k=0}^{n-1} \mu_{\ell_\alpha}(h_{\delta}\cdot\vf\circ F^k_\ve)+\cO( (\delta+\boldsymbol\epsilon\|h_\delta\|\nc0)\|\vf\|_{\cC^0}).
\]
In general the measures $\mu_{\ell,\delta}(\vf)=\mu_\ell(h_{\delta}\vf)$ are not standard pairs because the derivative of the density might be too big, however their push-foward for large enough times will eventually be described by standard families \cite[Proposition 5.2]{DeL2}. Thus, taking first the limit for $n\to \infty$, then $\boldsymbol\epsilon\to 0$ and, finally, $\delta\to 0$, we see that $\tilde \mu\in  \cM_{\operatorname{sp}}(F_\ve)$, as claimed.

The above imply that if $\mu$ is an extremal point of $\cM_{\operatorname{sp}}(F_\ve)$, then it must be ergodic. If not, then there exists an invariant set $A\subset \bT^2$, $\mu(A)\not\in\{0,1\}$. We can then define the probability measures $\mu_1(\vf)=\mu(A)^{-1}\mu(\Id_A \vf)$ and $\mu_2(\vf)=[1-\mu(A)]^{-1}\mu(\Id_{A^c}\vf)$. By the previous arguments $\mu_i\in \cM_{\operatorname{sp}}(F_\ve)$, but this is impossible since $\mu$, being a convex combination of the $\mu_i$, would not be an extremal point, contrary to the hypothesis.
We have thus seen that the extremal points are ergodic, hence they provide the ergodic decomposition of the measures in $\cM_{\operatorname{sp}}(F_\ve)$, as claimed.
\end{proof}

Let $\mu\in\cM_{\operatorname{sp}}(F_\ve)$, then, for each $\boldsymbol\epsilon>0$, \eqref{eq:mu-fam-app} and \cite[Proposition 9.7, Lemma 11.3]{DeL1} imply that, denoting $h(\cdot,\theta)$ the density of $\mu_\theta$,
\[
\begin{split}
\mu(\vf)=\mu(\vf\circ F_\ve^n)=&\sum_{\alpha\in\alphaset}\fm_\alpha\int_{\bT} h(x,\theta^*_{\ell_\alpha})\vf(x,\theta^*_{\ell_\alpha})dx\\
&+\cO([\boldsymbol\epsilon+\ve n]\,[ \|\vf\|_{\cC^0}+\|\partial_\theta\vf\|_{\cC^0}]+e^{-\const n}\|\vf\|_{\cC^1}).
\end{split}
\]
Next, for all $I_m=[m\ve,(m+1)\ve)$, define $\cI_m=\{\alpha\in\alphaset\;:\; \theta^*_{\ell_\alpha}=\mu_{\ell_\alpha}(G_{\ell_\alpha})\in I_m\}$ and consider a measure
$\nu$ on $\bT^1$ such that $\nu(I_m)=\lim_{\boldsymbol\epsilon\to 0}\sum_{\alpha\in\cI_m}\fm_\alpha$. We can then write
\begin{equation}\label{eq:phys-strut}
\mu(\vf)=\int_{\bT^2}\vf(x,\theta) h(x,\theta) dx\,\nu(d\theta)+\cO(\ve \ln\vei \|\vf\|_{\cC^1}),
\end{equation}
where we have chosen $n=\Const \ln\vei$ and taken the limit $\boldsymbol\epsilon\to 0$. This reduces the problem of understanding the structure of the measure $\mu$ to the one of determining the measure $\nu$.

To gain some control on  $\nu$ we can repeat the same argument, but for the longer time $n=n_1+n_2=n_1+\Const \ln\vei$, with $n_1=\lfloor t\ve^{-1-\gamma}\rfloor$, for some $\gamma\geq 0$ and fixed $t$. We use again \eqref{eq:mu-fam-app} and call $\fkL_\alpha$ the standard pair that describes the push-forward at time $n_1$ of the standard pair $\ell_\alpha$,
\[
\mu(\vf)=\mu(\vf\circ F_\ve^{n})=\sum_{\alpha\in\alphaset}\fm_\alpha\sum_{\ell\in\fkL_\alpha}\fm_\ell\mu_\ell(\vf\circ F_\ve^{n_2})+\cO(\boldsymbol\epsilon \|\vf\|_{\cC^0}).
\]
By Corollary \ref{cor:third-step} it follows that for $\alpha\in \cI_m$, the standard family $\fkL_\alpha$ is made of standard pairs distributed as the process $\Beta(n)$. More precisely, let $p(\theta,t,\Beta(0))$ the probability distribution of such a $\Beta(n)$, and note that $|\partial_\theta p|+|\partial_{\Beta(0)} p|\leq \Const \ve^{-\frac 12}p$, then \cite[Lemma 11.3]{DeL1} and Corollary \ref{cor:third-step} imply, for each $\beta>0$,
\[
\mu(\vf)=\int_{\bT^3}p(\theta,t\ve^{-\gamma},\theta')\vf(x,\theta) h(x,\theta) dx\,\nu(d\theta')d\theta+\cO(\ve^{1/2-\gamma-\beta}\|\vf\|_{\cC^0}+\ve \ln\vei \|\vf\|_{\cC^1}).
\]
This, together with \eqref{eq:phys-strut}, yields, for all $\bar\vf\in\cC^1$,
\begin{equation}\label{eq:phys-strut-2}
\begin{split}
\int_{\bT}\bar\vf(\theta)\nu(d\theta)=&\int_{\bT^2}p(\theta,t\ve^{-\gamma},\theta')\bar \vf(\theta)  \nu(d\theta')d\theta\\
&+\cO(\ve^{1/2-\gamma-\beta}\|\bar\vf\|_{\cC^0}+\ve \ln\vei \|\bar\vf\|_{\cC^1}).
\end{split}
\end{equation}
To use effectively the above equation it is convenient to consider separately the two main cases.

\subsection{Rotations}\label{sec:rotation}
We consider first the case in which $\bar\omega$ has no zeroes. In such a case the averaged equation has also the unique invariant measure $\bar\omega(\theta)^{-1} d\theta$. So, if we let $T$ be the period, we can use \eqref{eq:phys-strut-2}, with $t\leq T$ and $\gamma=0$, as
\begin{equation}\label{eq:rotation-meas-1}
\begin{split}
\int_{\bT}\bar\vf(\theta)\nu(d\theta)=&\frac 1T\int_0^T\int_{\bT^2}p(\theta,t,\theta')\bar \vf(\theta)  \nu(d\theta')d\theta dt\\
&+\cO(\ve^{1/2-\beta}\|\bar\vf\|_{\cC^0}+\ve \ln\vei \|\bar\vf\|_{\cC^1}).
\end{split}
\end{equation}
%Using the function $h_{\theta'}$ introduced in the proof of Lemma \ref{lem:conver-no-zero}, with the choice $\theta_0=\theta'$, and recalling \eqref{eq:eds Psi} we have
%\[
%\begin{split}
%\int_{\bT^2}p(\theta,t,\theta')\bar \vf(\theta)  \nu(d\theta')d\theta&=\int_{\bT}\bE(\bar\vf(\Beta(t))\;|\;\Beta(0)=\theta')\nu(d\theta')\\
%&=\int_{\bT}\bE(\bar\vf\circ h_{\theta'}^{-1}(\psi(t))\;|\;\Psi(0)=0)\nu(d\theta')\\
%&=\int_{\bT}\bar\vf\circ h_{\theta'}^{-1}(t)\nu(d\theta')+\cO(\ve^{\frac 12-\beta}\|\bar\vf\|_{\cC^0}).
%\end{split}
%\]

By Lemma \ref{lem:second-step} we have
\[
\begin{split}
\int_{\bT^2}p(\theta,t,\theta')\bar \vf(\theta)  \nu(d\theta')d\theta &=\int_{\bT^2}  \frac{e^{-(\bar\theta(t,\theta')-\theta)^2/(2\Var_t^2\ve)}}{\Var_t\sqrt{2\pi\ve}}\bar \vf(\theta)  \nu(d\theta')d\theta +\cO(\ve^{\frac 12-\beta}\|\bar\vf\|\nc0) \\
&=\int_{\bT}  \bar \vf(\bar\theta(t,\theta'))  \nu(d\theta') +\cO(\ve^{\frac 12-\beta}\|\bar\vf\|\nc1).\\
\end{split}
\]
Substituting the above in \eqref{eq:rotation-meas-1} yields
\[
\begin{split}
\int_{\bT}\bar\vf(\theta)\nu(d\theta)&=\int_{\bT}  \frac 1T\int_0^T\bar \vf(\bar\theta(t,\theta')) dt \nu(d\theta') +\cO(\ve^{\frac 12-\beta}\|\bar\vf\|\nc1)\\
&=\int_{\bT} \frac{\bar \vf(\theta)}{\bar\omega(\theta)}d\theta+\cO(\ve^{\frac 12-\beta}\|\bar\vf\|\nc1).
\end{split}
\]
Thus all the elements of $\cM_{\operatorname{sp}}(F_\ve)$, and hence also the eventual physical measures, are very close to the invariant measure
of the averaged system. More precisely, we have proven:

\begin{prop}\label{prop:srb-rot} For each $\beta>0$, if $\zeroes=0$, then for each $\mu\in\cM_{\operatorname{sp}}(F_\ve)$ and $\vf\in\cC^1(\bT^2,\bR)$,
\[
\mu(\vf)=\int_{\bT^2} \vf(x,\theta) \frac{h(x,\theta)}{\bar\omega(\theta)}  dx\,d\theta+\cO(\ve^{1/2-\beta} \|\vf\|_{\cC^1}).
\]
\end{prop}

The above result is good enough to compute the leading contribution to the Lyapunov exponent by arguing similarly to what we do in section \ref{sec:lyap} for the case in which there are sinks. However it does not suffices to investigate the mixing properties of the physical measure. In fact, Lemma \ref{lem:conver-no-zero} tells us that in the time $\ve^{-\gamma}$, $\gamma\leq 1/2$, the process $\Beta$ is still far from equilibrium.

We conjecture that in this case there exists a unique physical measure which mixes with speed $\ve^2$; but to prove such a result following the present strategy it would be necessary to improve the error in  \cite[Theorem 2.8]{DeL1}. More precisely we would need the to compute explicitly the first term in the Edgeworth expansion.

\subsection{Sinks}
Suppose that $\bar\omega$ has $2\zeroes$ non degenerate zeroes.
In such a case, Proposition \ref{prop: close to gaussians} and equation \eqref{eq:phys-strut-2}, choosing $\gamma\in (0,1/4)$, imply that there exists positive constants $\{c_i\}_{i=1}^n$ such that
\[
\int_{\bT}\bar\vf(\theta)\nu(d\theta)=\sum_{i=1}^nc_i\int_{\bT}\cG_i^{\ve}(\theta)\bar \vf(\theta)  d\theta+\cO(\ve^{1/2-2\gamma}\|\bar\vf\|_{\cC^0}+\ve \ln\vei \|\bar\vf\|_{\cC^1}).
\]
To compute the constant $c_i$ one must look at times longer than the metastability time scale. Indeed, by the large deviations results in \cite{DeL1} one can compute the probability to go from one sink to another. This analysis will show that, generically, there is a $c_i$ that is exponentially larger than the other, hence the invariant measure will look like a gaussian centred on the winning sink.  We will not pursue this issue further as it is not needed for our present discussions. For future reference, let us collect the result so far obtained.

\begin{prop}\label{prop:srb} If $\mu\in\cM_{\operatorname{sp}}(F_\ve)$, and $\zeroes>0$, then there exists $\bar c=\{c_k\}_{i=1}^\zeroes$, $c_k\geq 0$, $\sum_kc_k=1$, such that, for each $\vf\in\cC^1(\bT^2,\bR)$,
\[
\mu(\vf)=\sum_{k=1}^\zeroes c_k \int_{\bT^2} \vf(x,\theta) h(x,\theta) \cG^\ve_k(\theta) dx\,d\theta+\cO(\ve^{1/2-2\gamma}\|\vf\|_{\cC^0}+\ve \ln\vei \|\vf\|_{\cC^1}),
\]
where $\cG^\ve_k$ is a Gaussian distribution with mean $\theta_{k,-}$ and variance $\frac{ \bVar^2(\theta_{k,-})}{2 \bar\omega'(\theta_{k,-})\ve}$.
\end{prop}
Remark that the above proposition essentially provides an effective formula for computing the Lyapunov exponent for measures in $\cM_{\operatorname{sp}}(F_\ve)$, as we will see in the next section.

In \cite{DeL2} it is proven that if the measures in $\cM_{\operatorname{sp}}(F_\ve)$ have negative Lyapunov exponents, then there exists finitely many physical measures, and a checkable criteria for uniqueness is provided. A similar result for non negative Lyapunov exponents is missing, although, as already mentioned, in \cite{Tsujii-acta} it is proven that the physical measures exist generically.

\section{Lyapunov exponents}\label{sec:lyap}
In the previous section we have see that all the invariant measures obtained by the pushforward of a standard pair, must be $\ve^{\beta_1}$ close to each other in the $(\cC^1)'$ topology. Hence they may differ substantially only at a scale smaller than $\ve$. It remains open the question if the SRB measure always exists and if it is unique or not. To discuss such issues it seems necessary to have some information on the Lyapunov exponent of the central foliation. Recall from \cite[Section 3]{DeL2} that the $n$-step central direction is defined by
\[
\deh_p F_{\ve}^n(\stable_n(p),1)=\mu_n(0,1)
\]
from which it follows
\begin{equation}\label{eq:central-d}
\begin{split}
&\stable_n(p)=\Xi_p(\stable_{n-1}(F_\ve(p)))=\frac{[1+\ve\partial_\theta\omega(p)]\stable_{n-1}(F_\ve(p))-\partial_\theta f(p)}{\partial_x f(p)-\ve\partial_x\omega(p)\stable_{n-1}(F_\ve(p))}\\
&\mu_n(p)=\prod_{k=0}^n\left[ 1+\ve(\partial_x\omega(p_k)\stable_{n-k}(p_k)+\partial_\theta\omega(p_k))\right],
\end{split}
\end{equation}
where $p_k=F_\ve^k(p)$.
Note that, for $\ve$ small enough, there exists $K>0$ and $\sigma\in(0,1)$, such that, for each $p\in\bT^2$, $\Xi_p([-K,K])\subset [-K,K]$ and  $\sup_{\stable\in [-K,K]}|\Xi_p'(\stable)|\leq \sigma$. From this it follows that there exists $\hat\stable(p)$ such that
\[
|\hat\stable(p)-\stable_n(p)|\leq \sigma^n.
\]
The 1-dimensional line field $(1,\hat s(p))$ is called the central distribution which we denote~$E^c(p)$. It is known to be $F_\ve$-invariant and continuous in~$p$.

Then, setting
\[
\hat\mu_n(p)=\prod_{k=0}^n\left[ 1+\ve(\partial_x\omega(p_k)\hat\stable(p_k)+\partial_\theta\omega(p_k))\right],
\]
we have
\begin{equation}\label{eq:lyap-est-0}
|\mu_n(p)-\hat\mu_n(p)|\leq \Const \ve\hat\mu_n(p).
\end{equation}
We thus have that the central Lyapunov exponent is given by the ergodic average
\[
\chi_c(p)=\lim_{n\to\infty}\frac 1n \ln\mu_n=\lim_{n\to\infty}\frac 1n \ln\hat\mu_n.
\]
Next, by Lemma \ref{lem:ergodicity}, we can restrict ourselves to considering only  $\mu\in \cM_{\operatorname{sp}}(F_\ve)$ ergodic.
Then, the Birkhoff ergodic theorem imply that $\mu$ almost surely
\begin{equation}\label{eq:lyapunov}
\chi_c(p)=\mu\left(\ln \left[ 1+\ve(\partial_x\omega\cdot \hat\stable +\partial_\theta\omega )\right]\right).
\end{equation}
Taking the limit $n\to \infty$ in \eqref{eq:central-d} we have
\[
\hat\stable(p)=-\sum_{k=0}^\infty\frac{\partial_\theta f(F_\ve^k(p))}{\prod_{j=0}^k\partial_x f(F_\ve^j(p))}+\cO(\ve)=\stable_0(p)+\cO(\ve).
\]
Next, we obtain an even more explicit expression. Indeed, by \cite[Lemma 4.1]{DeL1}, for each $k\leq n\leq \Const \sqrt\ve$, we can write, for $p=(x_0,\theta_0)$,
\[
\begin{split}
&p_k=(f_{\theta_0}^k(Y_n(x_0)),\theta_0)+\cO(\ve k)\\
&\left\|1-Y_n'\right\|\leq\Const \ve n^2.
\end{split}
\]
We can then choose $n=C\ln\ve^{-1}$, for $C$ large enough, hence
\[
s_0(p)=-\sum_{k=0}^{C\ln\ve^{-1}}\frac{\partial_\theta f(f_{\theta_0}^k(Y_n(x_0)),\theta_0)}{(f_{\theta_0}^k)'(Y_n(x_0))}+\cO(\ve(\ln\ve^{-1})^3)
\]
Then, by Proposition \ref{prop:srb} and \eqref{eq:lyapunov}, we have
\[
\begin{split}
\ve^{-1}\chi_c(p)&=\mu\left(\partial_x\omega\cdot \stable_0 +\partial_\theta\omega\right)+\cO(\ve)\\
&=\sum_jc_j\mu_{\theta_{j,-}}\left(\partial_x\omega(\cdot,\theta_{j,-})\cdot \stable_0(\cdot,\theta_{j,-}) +\partial_\theta\omega(\cdot,\theta_{j,-})\right)+\cO(\sqrt\ve),
\end{split}
\]
for some $c_j > 0$, $\sum c_j = 1$, to be determined. Note that $\mu_{\theta_j,-}$ is absolutely continuous with respect to Lebesgue and its density $h_j$ is in $\cC^1$.
Thus we can write
\[
\begin{split}
&\mu_{\theta_{j,-}}\left(\partial_x\omega(\cdot,\theta_{j,-})\cdot \stable_0(\cdot,\theta_{j,-})+\partial_\theta\omega(\cdot,\theta_{j,-})\right)\\
&=\int h_j\circ Y_n^{-1}(x)\left[\partial_x\omega(Y_n^{-1}(x),\theta_{j,-})\cdot \stable_0(Y_n^{-1}(x),\theta_{j,-})+\partial_\theta\omega(Y_n^{-1}(x),\theta_{j,-})\right]\\
&\quad+\cO(\ve(\ln\ve^{-1})^2)\\
&=\int h_j(x)\left[\partial_x\omega(x,\theta_{j,-})\cdot \stable_0(Y_n^{-1}(x),\theta_{j,-})+\partial_\theta\omega(x,\theta_{j,-})\right]+\cO(\ve(\ln\ve^{-1})^2)\\
&=-\sum_{k=0}^{\infty}\int h_j(x)\left[\partial_x\omega(x,\theta_{j,-})\frac{\partial_\theta f(f_{\theta_{j,-}}^k(x),\theta_{j,-})}{(f_{\theta_{j,-}}^k)'(x)}\right]+
\mu_{\theta_{j,-}}\left(\partial_\theta\omega(\cdot,\theta_{j,-})\right)\\
&\quad+\cO(\ve(\ln\ve^{-1})^3).
\end{split}
\]
We arrive then to the rather explicit formula
\begin{equation}\label{eq:lyap-final}
\begin{split}
\ve^{-1}\chi_c(p)=&\sum_jc_j\mu_{\theta_{j,-}}\left(\partial_\theta\omega(\cdot,\theta_{j,-})\right)\\
&-\sum_jc_j\sum_{k=0}^{\infty}\mu_{\theta_{j,-}}\left(\partial_x\omega(\cdot,\theta_{j,-})\frac{\partial_\theta f(f_{\theta_{j,-}}^k(\cdot),\theta_{j,-})}{(f_{\theta_{j,-}}^k)'}\right)+\cO(\sqrt\ve).
\end{split}
\end{equation}
For a measure $\mu\in \cM_{\operatorname{sp}}(F_\ve)$ let $R_\mu(F_\ve)$ be the set of points in the support of $\mu$ for which the Lyapunov exponent exists. Then our argument shows that \eqref{eq:lyap-final} holds for each $p\in \cup_{\mu\in \cM_{\operatorname{sp}}} R_\mu(F_\ve)=:R(F_\ve)$.

\subsection{Skew product case}\label{ss:skew-products-lyap}
In the special case when the map~$F_\ve$ is a skew product
$$
F_\ve(x,\theta) = (f(x), \theta+\ve \omega(x,\theta)),
$$
all the terms~$\partial_\theta f$ in~\eqref{eq:lyap-final} are zero and thus the formula~\eqref{eq:lyap-final} reduces to
$$
\ve^{-1}\chi_c(p)=\sum_jc_j\mu_{\theta_{j,-}}\left(\partial_\theta\omega(\cdot,\theta_{j,-})\right) +\cO(\sqrt\ve).
$$
In addition, we have $c_j > 0$ and $\partial_\theta\omega(x,\theta_{j,-}) < 0$ for all $j$ and any $x$. Thus $\chi_c(p) < 0$ for $\ve$ small enough.

\subsection{A counterintuitive example}
\label{ss:desimoi-example}
We discuss in detail an example introduced, but not conclusively studied, in \cite{DeL1}.
Let $\ell\in\bN$, $\ell > 1$, $\alpha\in\bR$, $\beta>0$ and consider the family
\begin{equation}\label{eq:expanding-contracting}
  F_\ve(x,\theta)=(\ell x + \sin(2\pi\theta)\left[\alpha\sin(2\pi
    x)+\beta\sin(2\ell\pi x) \right], \theta + \ve\cos(2\pi x))\mod 1.
\end{equation}
In the above examples $\omega(x,\theta)=\cos(2\pi x)$ does not depend on $\theta$.
In \cite{DeL1} is computed
\[
  \bar\omega'(\theta)=
 \sum_{k=1}^\infty\int_{\bT}(\omega\circ f_\theta^k(x))'\frac{\partial_\theta
    f(x,\theta)}{f'_\theta(x)}\rho_\theta(x)\deh x.
\]
Observe that if $\theta=0$ or $\theta=1/2$ (so that $\sin(2\pi\theta)=0$), then
$f_\theta(x)=\ell x$, thus $\mu_\theta=\Leb$, $\bar\omega(\theta)=0$ and
\[
\bar\omega'(0)= -{(2\pi)^2}\sum_{k=1}^\infty\ell^{k-1}\int_\bT\sin(2\ell^k\pi x)[\alpha\sin(2\pi
  x)+\beta\sin(2\ell\pi x)]=
  -2\pi^2\beta.
\]
Then $\theta=0$ is a sink for the averaged dynamics and it turns out to be the only one.\footnote{ We refrain from proving it but it is not very hard to check it numerically.} Accordingly, \eqref{eq:lyap-final} yields, for all $p\in R(F_\ve)$,
\[
\begin{split}
\ve^{-1}\chi_c(p)&=4\pi^2\sum_{k=0}^{\infty}\int_{\bT}\sin(2\pi x)\frac{\alpha\sin(2\pi \ell^k x)+\beta\sin(2\pi\ell^{k+1} x)}{\ell^k}+\cO(\sqrt\ve)\\
&=2\pi^2\alpha+\cO(\sqrt\ve).
\end{split}
\]
As suggested in \cite{DeL1}, we thus see that,  for $\alpha>0$, the central Lyapunov exponent is positive, although the average dynamics tends to concentrate the motion in a very small neighbourhood of zero. This seems to be counterintuitive and has an interesting implication that we are going to discuss in the next section.

In addition, $\chi_c > 0$ also holds for small perturbations of~\eqref{eq:expanding-contracting}. Indeed, the locus where the most of mass of the physical measure sits is a~$\sqrt\ve$-band around a zero of~$\bar\omega(\theta)$. This locus depends on~$F_\ve$ continuously. The derivative~$DF|_{E^c}$ in the central direction depends on~$F_\ve$ continuously, too. Thus:

\begin{prop}	\label{prop:lyap-positive}
For any~$l \in \bN$, $l > 1$, and $\alpha, \beta>0$ there exists $\ve>0$ and a $C^1$-open set $\cU_{\ve}$, that contains the maps $F_{\ve'}$ defined in~\eqref{eq:expanding-contracting} with the given $l,\alpha,\beta$ for all $0\leq \ve'\leq\ve$, such that for any~$F \in \cU_{\ve}$ and $p\in R(F)$ we have $\chi_c(p) > 0$.
\end{prop}
In particular, if $F\in\cU_\ve$ has a physical measure, then it must have positive Lyapunov exponents.

\section{The central foliation}\label{sec:central-fol}

In Section~\ref{sec:lyap} we have seen that, contrary to naive intuition, it is possible that the central Lyapunov exponent~$\chi_c$ is positive, despite having a statistical sink. To make things worse, we prove below that~$F_\ve$ has an invariant foliation made of smooth compact leaves tangent to the central distribution. If~$\chi_c > 0$, these leaves have to expand in average but at the same time their length is uniformly bounded.

The reason why this is not contradictory  is that the center foliation fails to be absolutely continuous.
This means that, despite each leaf being individually smooth, the foliation as a whole is very wild. This situation is strange but known to happen, see the papers of Ruelle, Shub and Wilkinson~\cite{SW,  RW} where they presented an open set of volume preserving partially hyperbolic systems with non absolutely continuous central foliation for a perturbation of the product of an Anosov map by an identity map on the circle. This behaviour was later observed in many other partially hyperbolic systems, see \cite{Baraviera2003, Hirayama2007, Ponce2012, Ponce2013, Saghin2009, Varao2013}.

In our class of dynamical systems, we find a similar phenomenon. Namely, let~$\cU_{\ve}$ be given by Proposition~\ref{prop:lyap-positive}. Then for every~$F \in \cU_{\ve}$ we have~$\chi_c > 0$.
\begin{thm}\label{thm:foli-main}
For every map~$F$ from~$\cU_{\ve}$
\begin{enumerate}
\item\label{i:central-integrable} the central distribution~$E^c$ is uniquely integrable to a $C^1$ foliation~$\sW^c$;
\item\label{i:central-smooth} if, in addition, $F$ has $j$-pinching, $j \ge 1$, see Subsection~\ref{ss:central-foli-exists}, then~$\sW^c$ is $C^j$;
\item\label{i:cpt-leaves} every leaf~$W \in \sW^c$ is diffeomorphic to a circle and of uniformly bounded length;
\item\label{i:foli-non-ac} if $F$ has a physical measure, then $\sW^c$ is not absolutely continuous.
\end{enumerate}
\end{thm}

\begin{rem}
As already explained, according to Tsujii \cite{Tsujii-acta}, the existence of the physical measure is generic and hence it holds generically for $F \in \cU_{\ve}$. Accordingly, the above Theorem implies that generically $\sW^c$ is not absolutely continuous. In fact, it is quite possible that $\sW^c$ is always not absolutely continuous as it should be possible to extend to this, non invertible, case \cite[Theorem A]{ABV00} and hence show that there is always at least one physical measure.
\end{rem}

For possible future use we will prove many of the above results in much larger generality than stated in Theorem \ref{thm:foli-main}. Let us start by describing such a more general setting.

By a $k$-dimensional \emph{$C^r$ foliation}~$\sW$ of $M$, $r \ge 1$, we mean a partition of $M$ into $k$-dimensional, complete, connected~$C^1$ submanifolds $W(z) \ni z$, called \emph{leaves}, which depend continuously on~$z$. Let~$D^n$ denote the open unit ball in~$\bR^n$. For each point~$z \in M$ there is a \emph{coordinate chart} (or \emph{foliation box})~$(U, \phi)$ at $z$: a neighborhood~$U \ni z$ and a homeomorphism~$\phi\colon D^k \times D^{m-k} \to U$ such that for each~$p \in D^{m-k}$, the set~$W_U(\phi(0,p)) = \{ \phi(z,p) \}_{z \in D^k}$, called the \emph{local leaf}, is contained in~$W(\phi(0,p))$ and $\phi(\cdot,p) \colon D^k \to W(\phi(0, p))$ is a $C^r$ diffeomorphism which depends continuously on $p \in D^{m-k}$ in $C^1$ topology.

Given a foliation~$\sW$, denote by~$d_W$ the distance along the leaf~$W \in \sW$,  by~$W(z)$ the leaf passing through~$z \in M$, and by~$W_\delta (z)$ the ball of radius $\delta$ in~$W(z)$ centered in~$z$. A foliation of a simply connected Riemannian manifold~$\tilde M$ is called \emph{quasi-isometric}~\cite{Fenley1992} if there are $a, b > 0$ such that for any $z_1, z_2 \in W(z_1)$, holds
$$
d_W(z_1, z_2) \le a \cdot d(z_1,z_2) + b.
$$

A foliation~$W(z)$ is tangent to a distribution~$E(z)$ if for every $z \in M$ we have $T_z W(z) = E(z)$. A distribution $E$ is called \emph{integrable} if there exists a foliation tangent to~$E$, and \emph{uniquely integrable} if such foliation is unique.
The existence theorem for solutions of ODEs implies that for every continuous distribution~$E$, $\dim E = 1$, and for every~$z \in M$ there exists a local curve~$\gamma \ni z$ tangent to $E$. However, this $\gamma$ may not be unique and thus there may be no way to construct a global foliation of these curves, see for instance~\cite{ArXiv-Hertz2014}.
To assert~$\gamma$ is unique the distribution must have greater regularity, such a Lipshitz.
For~$\dim E \ge 2$ the classic Frobenius Theorem indicates that even infinitely smooth distributions may fail to be integrable, let alone uniquely integrable. Thus the question of integrability of the central distribution is very important in the theory of partially hyperbolic dynamical systems.

The analogue of uniqueness of solutions of ODEs in the world of distributions is the following property. A continuous $k$-dimensional distribution~$\sW$ is called \emph{locally uniquely integrable} if for each $z \in M$ there are $k$-dimensional $C^1$-submanifold~$W_{loc}(z)$ and~$\alpha(z) > 0$ such that every piecewise $C^1$ curve $\sigma \colon [0,1] \to M$ satisfying (i) $\sigma(0) = z$, (ii) $\dot \sigma(t) \in E(\sigma(t))$ for $t \in [0,1]$, and (iii) length$(\sigma) < \alpha(z)$, is contained in~$W_{loc}(z)$. Obviously, if~$\sW$ is locally uniquely integrable, then it is integrable and the integral foliation is unique. In addition, we say a distribution is \emph{$C^r$ locally uniquely integrable} if every~$W_{loc}(z)$ is a $C^r$ submanifold, $r \ge 1$.

Now let~$F\colon M \to M$ be a $C^1$ local diffeomorphism, not necessary 1-1 globally. We say that a foliation~$\sW$ is \emph{invariant} under $F$ if for every sufficiently small local leaf~$W$ of $\sW$ its image~$F(W)$ is also a local leaf of~$\sW$. Obviously, if $\sW$ is invariant under $F$, then its tangent distribution~$T\sW$ is invariant under $DF$. The converse is also true if $T\sW$ is uniquely integrable.

For every $F$-invariant measure~$\mu$ and every $F$-invariant foliation~$\sW$ the \emph{leafwise volume Lyapunov exponent} $\chi_{\sW}$ is well-defined for $\mu$-a.e. $z \in M$:
$$
\chi_{\sW}(z) = \lim_{n\to\infty} \frac1n \log \left| \det DF^n|_{T_z\sW} (z) \right|.
$$
If $F$ is a partially hyperbolic endomorphism, see Subsection~\ref{ss:central-foli-exists}, and its central distribution $E^c$ is uniquely integrable to~$\sW^c$, then $\chi_{\sW} = \chi_c$, where $\chi_c$ is the central Lyapunov exponent in case~$\dim E^c = 1$ and the sum of the central Lyapunov exponents in case~$\dim E^c \ge 2$.

In the following subsections we prove some results which are not only sufficient to prove claims~\ref{i:central-integrable}--\ref{i:foli-non-ac} but go well beyond.

\subsection{Claims \ref{i:central-integrable}-\ref{i:central-smooth}: the central foliation exists and is unique}	\label{ss:central-foli-exists}

Let~$F \colon M \to M$ be a $C^l$ local diffeomorphism, perhaps non-invertible globally. Let~$\tF \colon \tM \to \tM$ be a lift of $F$ to the universal cover~$\tM$. The map~$\tF$ is a 1-1 local diffeomorphism and thus a global diffeomorphism. In this paper, we say $F$ is a \emph{partially hyperbolic endomorphism} if $\tF$ has a uniform dominated splitting with a strong unstable bundle: there are constants $0 < \lambda_1  \le \lambda_2 < \mu_1 \le \mu_2$, $\mu_1 > 1$, and $C \ge  1$ and distributions~$\tE^c(\tz)$, $\tE^u(\tz)$, called \emph{center} and \emph{unstable}, respectively, such that for every~$\tz \in \tM$
\begin{itemize}
\item $T_\tz \tM = \tE^c(\tz) \oplus \tE^u(\tz)$;
\item the distributions~$\tE^c$, $\tE^u$ are invariant under~$D\tF$;
\item $C^{-1} \lambda_1^n \| v^c \| \le \| D\tF^n (\tz) v^c \| \le C \lambda_2^n \| v^c \|$ for each $v^c \in \tE^c(\tz)$ and $n > 0$;
\item $C^{-1} \mu_1^n \| v^u \| \le \| D\tF^n (\tz) v^u \| \le C \mu_2^n \| v^u \|$ for each $v^u \in \tE^u(\tz)$ and $n > 0$;
\end{itemize}
Denote $r = \max\{j \in \{1,...,l\}\,|\,\lambda_2^j < \mu_1\}$, the latter inequality sometimes being called the \emph{$j$-pinching} condition.

In ~\cite[Section 3]{DeL2} is proved the existence of unstable and center invariant cone fields for~$F_\ve$ for every~$\ve \le \ve_0$. This implies that the maps in $\cU_{\ve}$ are partially hyperbolic endomorphisms.

As opposed to the central distribution~$\tE^c$, the unstable distribution~$\tE^u$ of a $C^l$ diffeomorphism is known to be $C^l$ uniquely integrable.

The following theorem, a version of Brin's~\cite[Theorem 1]{Brin2003} for endomorphisms, establishes a connection between the geometry of the unstable foliation and the integrability of the central distribution.

\begin{thm}\label{t:central-integrable}
Let~$F$ be a partially hyperbolic endomorphism of a compact manifold~$M$. Suppose the unstable foliation of the lift~$\tF$ is quasi-isometric in the universal cover~$\tM$. Then the distribution~$E^c$ is $C^r$ locally uniquely integrable; in particular, $F$ has a unique central foliation and it is $C^r$.
\end{thm}
We do not provide the proof explicitly as the proof of $C^1$ local unique integrability is literally identical to the proof of the unique integrability of~$E^{cs}$ there. The additional $C^r$ regularity of the leaves follows from~\cite[Chapter 1, Theorem 4.10]{Hasselblatt2006}, applied to the inverse limit system for $F$. This proves claim~\ref{i:central-smooth}.

The following lemma, a version of~\cite[Proposition  4]{Brin2003} gives an elegant sufficient condition for a foliation to be quasi-isometric. Again, the proof follows Brin's word for word.

\begin{lem}		\label{l:quasi-isom-criterion}
Let~$\sW$ be a $k$-dimensional foliation of the $m$-dimensional space~$\bR^m$. Suppose there is an~$(m-k)$-dimensional plane~$A$ such that~$T_{\tz} W(\tz) \cap A = \emptyset$ for each~$\tz \in \bR^m$. Then~$\sW$ is quasi-isometric.
\end{lem}

To prove claim~\ref{i:central-integrable} of Theorem~\ref{thm:foli-main} it is now sufficient to show

\begin{prop}
For any map $F\in\cU_{\ve}$ the unstable foliation~$\sW^u$ of the lift~$\tF$ to the universal cover~$\tM = \bR^2$ satisfies the assumption of Lemma~\ref{l:quasi-isom-criterion} with~$k=1$, $m=2$.
\end{prop}

\begin{proof}
The metrics on~$\bT^2$ and $\bR^2$ are flat and the connections are trivial. Thus we can trivially identify all the tangent spaces $T_z \bT^2$, $z \in \bT^2$, and $T_\tz \bR^2$, $\tz \in \bR^2$. The union of all possible~$\tE^u(\tz)$, $\tz \in \bR^2$, is a subset of the unstable cone for~$F$ and thus avoids the central cone for~$F$. Thus any direction within the central cone, including the vertical direction, works as $A$.
\end{proof}

\begin{rem}
This straightforwardly generalizes to the maps of form~\eqref{eq:f-eps} in any dimension. Thus all such maps have locally uniquely integrable central distributions.
\end{rem}

\subsection{Claim~\ref{i:cpt-leaves}: central leaves are compact and have uniformly bounded volume}	\label{ss:bounded-volume}
To prove claim~\ref{i:cpt-leaves}, we need more assumptions on~$M$ and $F$. Let $C_*>0$ some arbitrary, but fixed, constant and $M_1,M_2$ be compact Riemannian manifolds. Given $M = M_1 \times M_2$, let  $\bF_\ve(M_1,M_2)$,  $\ve \le \ve_0$ , be the set of a partially hyperbolic endomorphisms $F \colon M \to M$, $\|F\|_{\cC^2}\leq C_*$, of the form
\begin{equation}	\label{eq:f-eps-gen}
F \colon (x, \theta) \mapsto (f(x,\theta), \Omega(x,\theta)), \quad x \in M_1, \theta \in M_2, \quad \dist(\Omega(x,\theta), Id) \le \ve,
\end{equation}
and assume $f(\cdot,\theta)$ is strictly expanding in~$x$ for every~$\theta \in M_2$. Note that we have $\cU_\ve\subset \bF_\ve(\bT^1,\bT^1)$.

Clearly, there is~$\ve_0 > 0$ such that for every~$\ve \le \ve_0$ the set~$\bF_\ve(M_1,M_2)$ is made of partially hyperbolic endomorphism in the sense of Subsection~\ref{ss:central-foli-exists}. Then by Theorem~\ref{t:central-integrable}, every~$F\in\bF_\ve(M_1,M_2)$ has a unique smooth central foliation~$W^c(z)$.

\begin{thm}	\label{t:bounded-volume}
For every~$\ve \le \ve_0$ there is~$V > 0$ such that for any~$z \in M$
\begin{itemize}
\item $W^c(z)$ is homeomorphic to~$M_2$;
\item $\vol(W^c(z)) < V$.
\end{itemize}
\end{thm}

\begin{proof}
We prove it first for the special case $\ve=0$ where we have an explicit description of the central
foliation of each $F_0\in \bF_0(M_1,M_2)$. Recall that every smooth expanding map is structurally stable. Thus all the maps $f(\cdot, \theta)$ are conjugated. Let $h(\cdot,\theta)$ be the map that conjugates $f(\cdot, 0)$ with $f(\cdot, \theta)$. By definition, $h(x,0)=h(x,1)=x$.
The graph of~$h(x, \cdot)$ is a compact submanifold of~$M$ homeomorphic to~$M_2$.
Thus we have a continuous foliation of~$M$ by the graphs of~$h(x, \cdot)$. This foliation is invariant under~$F_0$ and must coincide with the central foliation~$\sW^c$ for $F_0$. In particular, the leaves of $\sW^c$ are compact smooth submanifolds of~$M$ homeomorphic to~$M_2$.

To prove Theorem~\ref{t:bounded-volume} for $\ve > 0$, we will use the structural stability of the foliations. Following~\cite{Bonatti2004}, we say the central foliation~$\sW^c_F$ of the map~$F$ is \emph{structurally stable} if, given any nearby $C^1$ map~$G$,
\begin{enumerate}
\item the central distribution of~$G$ uniquely integrates to the central foliation~$\sW^c_G$;
\item there exists a globally defined homeomorphism $h_G$ sending leaves of $\sW^c_F$ to leaves of $\sW^c_G$;
\item $h_G \circ F \circ h_G^{-1}$ is isotopic to $G$ along the leaves.
\end{enumerate}

\begin{prop}\label{prop:strut-stab} For any~$\ve \le \ve_0$ and $F\in\bF_\ve(M_1,M_2)$ the foliation~$(F, \sW^c_F)$ is structurally stable.
\end{prop}

\begin{proof}
We lift~$F$ and~$\sW_F^c$ to the universal cover of~$M$ and use Theorem~(7.1) from~\cite{HPS1977}.
\end{proof}

Since $\overline \bF(M_1,M_2)$, the $\cC^1$ closure of $\bF(M_1,M_2)$, is compact in the $\cC^1$ topology, we can use Proposition \ref{prop:strut-stab} to cover it with finitely many balls of structural stability. We conclude that there exists a globally defined homeomorphism $h$ sending leaves of $\sW^c_{F_0}$ to leaves of $\sW^c_{F}$ and $h \circ F_0 \circ h^{-1}$ is isotopic to $F$ along the leaves. In particular, the leaves of~$\sW^c_{F}$ are compact smooth submanifolds of~$M$ homeomorphic to~$M_2$.

Moreover, because at every point~$z \in M$ the central space~$E^c(z) \subset T_z M$ belongs to the same cone~$K^c = \{ (\xi, \Beta) \in T_z M_1 \oplus T_z M_2 \,|\, |\xi| \le \gamma^c |\Beta| \}$, this can be proven as in~\cite[Section 3]{DeL2}, by Pythagoras Theorem for every~$z \in M$ we have
$$
\vol (W(z)) \le \sqrt{1+\gamma^2} \cdot \vol(M_2).
$$
\end{proof}

The above proof yields an interesting by-product consequence for families of expanding maps which, although folklore, we couldn't find stated explicitly in the literature.
\begin{rem}
Let~$f_\theta\colon M \to M$, $\theta \in (-\theta_0, \theta_0)$, be a smooth family of expanding maps, $C^r$ jointly in~$x$ and $\theta$. Let~$h(x,\theta)$ be the conjugacy map as above. Then $h(x,\theta)$ is $C^r$ smooth in~$\theta$.
\end{rem}
\begin{proof}
Apply the above argument to the endomorphism $F(x, \theta) = (f_\theta (x), \theta))$. Note that~$F$ has $r$-pinching because~$\lambda_2 = 1$. This implies that the graphs of~$h(x, \cdot)$ are~$C^r$-smooth, see \cite{HPS1977}.
\end{proof}
Of course, as we will see in the Subsections \ref{subsec:no-abs} and \ref{ss:eps-0}, one cannot expect~$h(x,\theta)$ to be smooth in~$x$, or even absolutely continuous.

\subsection{Claim~\ref{i:foli-non-ac}: the central foliation is not absolutely continuous}\label{subsec:no-abs}
Let $M$ be a Riemannian manifold equipped with a continuous foliation~$\sW$.
Denote by~$\Leb$ the Lebesgue measure on $M$ coming from the Riemannian volume. It follows from the classic works of Rokhlin that for any foliation box\footnote{As defined after Theorem~\ref{thm:foli-main}}~$\cB$ there exists a disintegration of $\Leb|_\cB$ into the transversal measure~$\tilde \mu_\cB$ and leafwise conditional measures~$\nu_{W,\cB}$ defined for~$\tilde\mu$-almost every leaf disk~$W_\cB$ within the box. The measures coming from different boxes are equivalent on their common domain so we drop the index~$\cB$ for brevity.

Since every leaf~$W \in \sW$ is a smooth submanifold of~$M$, it has the induced Riemannian volume and the Lebesgue measure~$\Leb_W$ coming from it.
Regularity of~$\nu_W$ with respect to $\Leb_W$ is a good indicator of how nicely the foliation box around the leaf~$W$ is immersed in~$M$.
We say a foliation is \emph{absolutely continuous} if for Lebesgue almost every $z \in M$ the conditional measure~$\nu_{W(z)}$ is absolutely continuous with respect to~$\Leb_{W(z)}$. There are other definitions of absolute continuous foliations, see for instance~\cite{Hirayama2007}, but they are beyond the scope of this paper.

%Our strategy is to use the following non-invertible analogue of Theorem 1.3 from [Hirayama, Pesin].

\begin{thm}\label{thm:foliation-not-acim}
Let $F$ be a $C^2$ partially hyperbolic endomorphism of a compact smooth Riemannian manifold $M$. Assume that
\begin{enumerate}
%\item $\dim E^c = 1$;
%\item\label{i:smooth-srb} $F$ preserves an absolutely continuous ergodic measure~$\mu$;
\item\label{i:smooth-srb} $F$ has a physical measure~$\mu$;
%\item\label{i:bounded-volume} the central distribution $E^c$ is integrable to a foliation $W^c$ with leaves of bounded volume;
\item\label{i:foliation-exists-bounded-volume} $F$ has a $C^1$ invariant foliation $\sW$ with leaves of uniformly bounded volume;
\item\label{i:1-1} for each leaf~$W \in \sW$ the restriction~$F_W \colon W \to F(W)$ is a 1-1 map;
\item\label{i:positive-exp} the leafwise volume Lyapunov exponent~$\chi_{\sW}$ w.r.t.~$\mu$ is strictly positive;
\end{enumerate}
Then the foliation $\sW$ is not absolutely continuous.
%Moreover,
%if $\mu$ is ergodic and all the central exponents are simultaneously positive or negative, then
%the conditional measures induced by $\mu$ on leaves of $W^c$ are atomic.
\end{thm}
\begin{proof}%[Proof of Theorem~\ref{thm:foliation-not-acim}]
Assume that the foliation~$\sW$ is absolutely continuous.
Let~$\Lambda$ be the set of Lyapunov regular points for~$\mu$.
%Because~$\mu(\Lambda) = 1$ and $\mu \ll \Leb$,
Because~$\mu$ is physical,
we have $\Leb(\Lambda) > 0$. Thus there exists a set~$A$, $\Leb(A) > 0$, such that for every~$z \in A$ we have~$\Leb_z (W(z) \cap \Lambda) > 0$.
%It follows from the Poincar\'e recurrence theorem that there exists a Borel measurable set $R \subset A$ with $\mu(R) = \mu(A)$ such that every point $z \in R$ returns to $R$ infinitely many times under iterations of $F$.

%Fix $z \in R$ and denote by~$\tau_i$, $i \ge 1$, the $i$-th return time of $z$.
Fix any $z \in A$ and denote by $\Jac_\sW F^n$  the determinant of the restriction of $DF^n$ to $E^c$. Then for any~$n \ge 0$ for the volume of the leaf~$W(F^{n}(z))$, remembering assumption~\ref{i:1-1}, we can write
\[
\begin{split}
\vol W(F^n(z)) = &\int\limits_{W(F^n(z))} \, dm_{F^n(z)} =  \int\limits_{W(z)} \left| \Jac_\sW F^n \right| \, dm_z \ge
\int\limits_{W(z) \cap \Lambda} \left| \Jac_\sW F^n \right| \, dm_z\\
&= \int\limits_{W(z) \cap \Lambda} e^{n \cdot \frac1n \cdot \log | \Jac_\sW F^n |} \, dm_z .
\end{split}
\]
Then, by Jensen's inequality for~$e^x$,
\[
\vol W(F^n(z)) \ge e^{n \cdot \int\limits_{W(z) \cap \Lambda} \frac1n \log | \Jac_\sW F^n | \, dm_z}.
\]
Note that for every~$z' \in {W(z) \cap \Lambda}$ we have $\frac1n \log | \Jac_\sW F^n (z') | \to \chi_{\sW}$.
%, where $\chi_{\sW}$ is the sum of the leaf-wise Lyapunov exponents.
Thus by Fatou's lemma
$$
\lim_{n\to\infty} \int\limits_{W(z) \cap \Lambda} \frac1n \log |\Jac_\sW F^n| \, dm_z \ge \int\limits_{W(z) \cap \Lambda} \chi_{\sW} \, dm_z.
$$
By assumption~\ref{i:foliation-exists-bounded-volume}, the volume of the leaves is uniformly bounded, i.e, there exists~$C \in \bR$ such that for any~$z \in M$ we have $C \ge \vol W(z)$. Thus we can write
$$
C \ge
\lim_{n\to\infty} \vol W(F^n(z)) \ge
\lim_{n\to\infty} e^{n \cdot \int\limits_{W(z) \cap \Lambda} \chi_{\sW} \, dm_z} \ge \lim_{n\to\infty} e^{n \cdot \chi_{\sW} \cdot m_z ({W(z) \cap \Lambda})} =
+\infty,
$$
because $\chi_{\sW} > 0$ by assumption~\ref{i:positive-exp} of the theorem. This contradiction proves the theorem.
\end{proof}

\begin{proof}[{\bf Proof of Theorem  \ref{thm:foli-main}}]
The idea is to apply Theorem~\ref{thm:foliation-not-acim}, with~$\sW = \sW^c_{F}$, to the maps~$F \in \cU_{\ve}$.
We have thus to check the hypotheses of Theorem~\ref{thm:foliation-not-acim}.
The existence of a central foliation satisfying \ref{i:foliation-exists-bounded-volume} is established in Subsection~\ref{ss:bounded-volume}. We obviously have~\ref{i:1-1} for~$F_0$ because of the special structure of~$\sW^c_{F_0}$, recall Subsection~\ref{ss:bounded-volume}. Then  that for any~$F\in\cU_\ve$, $\ve \le \ve_0$, and~$W \in \sW^c_{F}$ we know that~$F|_W \colon W \to F(W)$ is a local diffeomorphism (thus, a covering) and, by Proposition \ref{prop:strut-stab}, is isotopic to a map, topologically conjugated to some~$F_0|_{W'}$ which is 1-1, where~$W' \in \sW^c_{F_0}$. Thus $F|_W$ is itself 1-1. Finally, \ref{i:positive-exp} follows from Proposition \ref{prop:lyap-positive}.
\end{proof}
Let us conclude the paper stating few interesting related facts.
\begin{cor}
Suppose a $C^2$ partially hyperbolic endomorphism $F \colon M_1 \times M_2 \to M_1 \times M_2$ is a skew product
$$
F(x,\theta) = (f(x), \theta + \ve \omega(x,\theta))
$$
and has an absolutely continuous ergodic invariant measure~$\mu$. Then its central Lyapunov exponent with respect to~$\mu$ is non-positive.
\end{cor}

\begin{proof}
Assumptions~\ref{i:smooth-srb}--\ref{i:1-1} follow from the skew product structure. The central foliation in this case is the collection of all~$\{x\} \times M_2$, $x \in M_1$, which is obviously absolutely continuous. But if we assume~$\chi_c > 0$ this would imply that central foliation is not absolutely continuous. Thus~$\chi_c \le 0$.
\end{proof}

This fits well within general knowledge in the area. In different settings, it known~\cite{Kleptsyn2014a} that a generic partially hyperbolic skew product with a non-invertible base dynamics has negative central volume Lyapunov exponent, which can only become zero in some degenerate cases but never above zero. Kleptsyn, Nalskii~\cite{KN} used a similar approach to prove that a generic random dynamical systems on the circle contracts the orbits. Both results are based on the fundamental Baxendale's~\cite{Bax} theorem for stochastic flows.

\begin{rem}
Note that the situation is different for diffeomorphisms, see~\cite{Hirayama2007}. In that setting it is sufficient to ask~$\chi_c \ne 0$ instead of~$\chi_c > 0$ to prove that the central foliation is not absolutely continuous.
\end{rem}

A final comment on the case $\chi_c = 0$. Consider on the one hand a rigid rotation skew product
$$
F_\ve (x, \theta) = (f(x), \theta+\ve\omega(x))
$$
which has the vertical circles as the absolutely continuous central foliation.
On the other hand a system of the form
$$
F (x, \theta) = (f(x, \theta), \theta)
$$
with a generic~$f(x, \theta)$, $\partial_x f > \lambda > 1$, which has a non-absolutely continuous central foliation (we prove it shortly in Subsection~\ref{ss:eps-0}). Thus, both possibilities can happen. Clearly, there is the need for further investigation if we want to understand the absolutely continuity of the foliation in this case.

\subsection{Non-absolute continuity for $\ve = 0$}	\label{ss:eps-0}

For the special case $\ve = 0$ the central Lyapunov exponent~$\chi_c = 0$ and thus the Theorems~\ref{thm:foli-main} and~\ref{thm:foliation-not-acim} do not apply. However, recall the classic result by Shub, Sullivan~\cite{Shub1985}:
\begin{thm} Let $2 \le r \le \omega$. If two orientation preserving expanding $C^r$ endomorphisms $f$ and $g$ of $\bT^1$ are absolutely continuously conjugate, then they are conjugate by a $C^r$ diffeomorphism.
\end{thm}

In particular, the multipliers of all the according periodic points of $f$ and $g$ must be the same. This is a degeneracy of codimension infinity. In the concrete family~\eqref{eq:expanding-contracting} the multiplier of the fixed point~$(0, \theta)$ non-trivially changes with~$\theta$.
Thus, for a generic~$F_0$ the conjugacy~$h(x,\theta)$ is not absolutely continuous in~$x$. This implies that

\begin{prop} The map $F_0$ generically has a non absolutely continuous central foliation~$\sW^c_0$.
\end{prop}
Such type of results go back at least to Katok, see \cite{milnor} for a discussion of the piecewise linear case.

\appendix

%%BIBLIOGRAPHY

\end{document}